\documentclass[12pt]{amsart}

\usepackage{amssymb, amsmath, amscd, newlfont, bbm, slashbox}
\usepackage[foot]{amsaddr}

\usepackage{hyperref} 
\hypersetup{
	colorlinks=true,        
	linkcolor={black},  
	citecolor={black},
	urlcolor={black}     
}
\usepackage{enumitem}
\usepackage{tikz-cd}
\usepackage{comment}
\usepackage{cite}

\newcommand{\spacedcdot}{{\,\cdot\,}}

\newcommand{\Q}{{\mathbb{Q}}}
\newcommand{\C}{{\mathbb{C}}}
\newcommand{\R}{{\mathbb{R}}}
\newcommand{\Z}{{\mathbb{Z}}}

\newcommand{\Lie}{{\mathrm{Lie}}}

\newcommand{\GL}{{\mathrm{GL}}}

\newcommand{\Sp}{{\mathrm{Sp}}}

\newcommand{\Diag}{{\mathrm{Diag}}}

\DeclareMathOperator{\vol}{{\mathrm{vol}}}
\DeclareMathOperator{\artanh}{{\mathrm{artanh}}}

\DeclareMathOperator{\tr}{{\mathrm{tr}}}

\newcommand{\calA}{{\mathcal{A}}}

\newcommand{\calD}{{\mathcal{D}}}

\newcommand{\calH}{{\mathcal{H}}}

\newcommand{\calU}{{\mathcal{U}}}

\newcommand{\calZ}{{\mathcal{Z}}}

\newcommand{\frakg}{{\mathfrak{g}}}

\usepackage[mathscr]{eucal}

\providecommand{\abs}[1]{\left\lvert#1\right\rvert}
\providecommand{\norm}[1]{\left\lVert#1\right\rVert}
\providecommand{\scal}[2]{\left<#1,#2\right>}

\newtheorem{theorem}{Theorem}[section]
\newtheorem{lemma}[theorem]{Lemma}
\newtheorem{proposition}[theorem]{Proposition}
\newtheorem{corollary}[theorem]{Corollary}

\theoremstyle{definition}

\theoremstyle{remark}

\newtheorem*{remark*}{Remark}

\newtheoremstyle{named}{}{}{\itshape}{}{\bfseries}{.}{.5em}{#1 \thmnote{#3}}
\theoremstyle{named}
\newtheorem*{namedtheorem}{Theorem}
\newtheorem*{namedcorollary}{Corollary}
\newtheorem*{namedproposition}{Proposition}

\numberwithin{equation}{section}

\title{ On a family of Siegel Poincar\'e series}

\author{Sonja \v Zunar}

\address{Faculty of Geodesy,
	University of Zagreb,
	Ka\v ci\'ceva 26,
	10000 Zagreb,
	Croatia}
\email{szunar@geof.hr}

\subjclass[2020]{11F03, 11F46}
\thanks{The author acknowledges the Croatian Science Foundation grant IP-2018-01-3628.}
\keywords{Siegel cusp forms, Poincar\'e series}

\begin{document}

\begin{abstract}
	Let $ \Gamma $ be a congruence subgroup of $ \Sp_{2n}(\Z) $.
	Using Poincar\'e series of $ K $-finite matrix coefficients of integrable discrete series representations of $ \Sp_{2n}(\R) $, we construct a spanning set for the space $ S_m(\Gamma) $ of Siegel cusp forms of weight $ m\in\Z_{>2n} $. We prove the non-vanishing of certain elements of this spanning set using Mui\'c's integral non-vanishing criterion for Poincar\'e series on locally compact Hausdorff groups. Moreover, using the representation theory of $ \Sp_{2n}(\R) $, we study the Petersson inner products of corresponding cuspidal automorphic forms, thereby recovering a representation-theoretic proof of some well-known results on the reproducing kernel function of $ S_m(\Gamma) $. Our results are obtained by generalizing representation-theoretic methods developed by Mui\'c in his work on holomorphic cusp forms on the upper half-plane to the setting of Siegel cusp forms of a higher degree.
\end{abstract}

\maketitle

\section{Introduction}

In this paper, we generalize representation-theoretic methods used by G.\ Mui\'c in his work on holomorphic cusp forms of integral weight on the upper half-plane \cite{muic10, muic12} to the setting of Siegel cusp forms of a higher degree. Employing said methods, we construct a spanning set for the space $ S_m(\Gamma) $ of Siegel cusp forms of degree $ n\in\Z_{>0} $ and weight $ m\in\Z_{>2n} $ with respect to any congruence subgroup $ \Gamma $ of $ \Sp_{2n}(\Z) $ and shed some light on certain properties of the constructed Siegel cusp forms.

To explain our results, we recall that the space $ S_m(\Gamma) $ consists of the holomorphic functions $ f $ on the Siegel upper half-space $ \calH_n=\left\{z=x+iy\in M_n(\C):z^\top=z,\ y>0\right\} $ such that $ \sup_{z\in\calH_n}\abs{f(z)}\det y^{\frac m2}<\infty $ and 
\[ \left(f\big|_m\gamma\right)(z):=\det(Cz+D)^{-m}f\left((Az+B)(Cz+D)^{-1}\right)=f(z) \]	
for all $ \gamma=\begin{pmatrix}A&B\\C&D\end{pmatrix}\in\Gamma $ and $ z\in\calH_n $. Equipped with the Petersson inner product
	\[ \scal{f_1}{f_2}_{\Gamma\backslash\calH_n}=\frac1{\abs{\Gamma\cap\left\{\pm I_{2n}\right\}}}\int_{\Gamma\backslash\calH_n}f_1(z)\,\overline{f_2(z)}\,\det y^m\,d\mathsf v(z), \]
where $ d\mathsf v(z)=(\det y)^{-n-1}\,\prod_{1\leq r\leq s\leq n}dx_{r,s}\,dy_{r,s} $, $ S_m(\Gamma) $ is a finite-dimensional Hilbert space. One of our main results is the following theorem.
\begin{namedtheorem}[\ref{thm:069}]
	Let $ \Gamma $ be a congruence subgroup of $ \Sp_{2n}(\Z) $, and let $ m\in\Z_{>2n} $. For every polynomial $ \mu\in\C[X_{r,s}:1\leq r,s\leq n] $, let $ f_{\mu,m}:\calH_n\to\C $,
	\[ 	f_{\mu,m}(z)=(2i)^{mn}\,\frac{\mu((z-iI_n)(z+iI_n)^{-1})}{\det(z+iI_n)^m}. \]
	Then, the Poincar\'e series 
	\[ P_{\Gamma}f_{\mu,m}=\sum_{\gamma\in\Gamma}f_{\mu,m}\big|_m\gamma \] 
	converges absolutely and uniformly on compact subsets of $ \calH_n $, and we have
		\[ S_m(\Gamma)=\left\{P_\Gamma f_{\mu,m}:\mu\in\C[X_{r,s}:1\leq r,s\leq n]\right\}.  \]
\end{namedtheorem}
	
Our proof of Thm.\ \ref{thm:069} uses the classical lift $ \Phi_m $ that takes every function $ f:\calH_n\to\C $ to a function $ F_f:\Sp_{2n}(\R)\to\C $,
\[ F_f(g)=\left(f\big|_mg\right)(iI_n). \]
The lift $ \Phi_m $ is known to map Siegel cusp forms in $ S_m(\Gamma) $ to cuspidal automorphic forms on $ \Gamma\backslash\Sp_{2n}(\R) $. Moreover, the space $ \Phi_m(S_m(\Gamma)) $ can be described explicitly in terms of the spectral decomposition of $ L^2(\Gamma\backslash\Sp_{2n}(\R)) $ (see Prop.\ \ref{prop:034} and \ref{prop:081}). By this description and a result of D.\ Mili\v ci\'c (see Lem.\ \ref{lem:023}), the space $ \Phi_m(S_m(\Gamma)) $ is spanned by the Poincar\'e series $ P_\Gamma\varphi=\sum_{\gamma\in\Gamma}\varphi(\gamma\spacedcdot) $ of certain $ K $-finite matrix coefficients $ \varphi $ of an antiholomorphic discrete series representation $ \pi_m^* $ of $ \Sp_{2n}(\R) $ (see Prop.\ \ref{prop:029}). In Prop.\ \ref{prop:057}\ref{prop:057:2}, we prove that these matrix coefficients are exactly the functions $ F_{f_{\mu,m}} $ using some classical Harish-Chandra's results and the well-known fact that the functions $ f_{\mu,m} $ constitute the subspace $ (H_m)_K $ of $ K $-finite vectors in a holomorphic discrete series representation $ \left(\pi_m,H_m\right) $ of $ \Sp_{2n}(\R) $.

In Sect.\ \ref{sec:081}, we study the non-vanishing of Poincar\'e series $ P_\Gamma F_{f_{\mu,m}} $ and (equivalently) $ P_\Gamma f_{\mu,m} $ using Mui\'c's integral non-vanishing criterion given by Lem.\ \ref{lem:042}. The main result of this section is Prop.\ \ref{prop:043}. Here, we only highlight its following corollary.

\begin{namedcorollary}[\ref{cor:081}]
	Let $ m\in\Z_{>2n} $ and $ \mu\in\C[X_{r,s}:1\leq r,s\leq n]\setminus\left\{0\right\} $. Then,
	\[ P_{\Gamma_n(N)}F_{f_{\mu,m}}\not\equiv0\quad\text{and}\quad P_{\Gamma_n(N)}f_{\mu,m}\not\equiv0\quad\text{for}\quad N\gg0. \] 
\end{namedcorollary}
 
\noindent The following proposition is an example of a more explicit non-vanishing result.

\begin{namedproposition}[\ref{prop:082}]
	Let $ m\in\Z_{>2n} $ and $ l\in\Z_{\geq0} $. Then:
	\begin{enumerate}[label=\textup{(\roman*)}]
		\item\label{prop:054:1a} $ P_{\Gamma_n(N)}F_{f_{\det^l,m}}\equiv0 $ and $ P_{\Gamma_n(N)}f_{\det^l,m}\equiv0 $ if one of the following holds:
		\begin{enumerate}[label=\textup{(V\arabic*)}]
			\item\label{enum:050:1a} $ N=1 $ and $ 4\nmid m+2l $
			\item\label{enum:050:2a} $ N=2 $ and $ 2\nmid m $.
		\end{enumerate}
		\item\label{prop:054:2a} Let us denote
		\[ A_t^+=\left\{x\in\R^n:t>x_1>\ldots>x_n>0\right\},\qquad t\in\R_{>0}. \]
		We define the functions $ \phi_{l,m}:A_1^+\to\R_{\geq0} $,
		\[ \phi_{l,m}(x)=\left(\prod_{r=1}^nx_r^{\frac l2}(1-x_r)^{\frac m2-n-1}\right)\prod_{1\leq r<s\leq n}(x_r-x_s), \] 
		and $ M:\Z_{>0}\to[0,1] $,
		\[ M(N)=\left(\sqrt{1+\frac{4n}{N^2}}+\sqrt{\frac{4n}{N^2}}\right)^{-2}, \]
		and let $ N_0(\det^l,m) $ be the smallest integer $ N\in\Z_{>0} $ such that
		\[ \int_{A_{M(N)}^+}\phi_{l,m}(x)\,dx>\frac12\int_{A_1^+}\phi_{l,m}(x)\,dx. \]
		Let $ N\in\Z_{>0} $ such that neither \ref{enum:050:1a} nor \ref{enum:050:2a} holds. If $ N\geq N_0(\det^l,m) $, then $ P_{\Gamma_n(N)}F_{f_{\det^l,m}}\not\equiv0 $ and $ P_{\Gamma_n(N)}f_{\det^l,m}\not\equiv0 $.
	\end{enumerate}
\end{namedproposition}

In Sect.\ \ref{sec:082}, using the representation theory of $ \Sp_{2n}(\R) $, we compute the Petersson inner products of the Poincar\'e series $ P_{\Gamma}F_{f_{1,m}} $ with cuspidal automorphic forms in $ \Phi_m(S_m(\Gamma)) $, thus extending Mui\'c's result \cite[Thm.\ 2-11]{muic12}, in which only the case when $ n=1 $ was considered. By applying the inverse of the lift $ \Phi_m $, we obtain 

\begin{namedcorollary}[\ref{cor:079}]
	Let $ \Gamma $ be a congruence subgroup of $ \Sp_{2n}(\Z) $, and let $ m\in\Z_{>2n} $. Then,
	\[ \scal f{P_\Gamma f_{1,m}}_{\Gamma\backslash\calH_n}=C_{m,n}\,f(iI_n),\qquad f\in S_m(\Gamma), \]
	where
	\[ C_{m,n}=2^{\frac{n(n+3)}2}\pi^{\frac{n(n+1)}2}\prod_{r=1}^n\frac{\Gamma\left(m-\frac{n+r}2\right)}{\Gamma\left(m-\frac{r-1}2\right)}. \]
\end{namedcorollary}

As an application of Cor.\ \ref{cor:079}, in Sect.\ \ref{sec:083} we easily recover the well-known formula for the Siegel cusp forms that, in a sense, comprise the reproducing kernel function for $ S_m(\Gamma) $ (see \cite[p.\ 30]{godement10} and \cite[the last formula]{godement6}; cf.\ \cite[Ch.\ 3, \S6, Thm.\ 1]{klingen}). In other words, we give a representation-theoretic proof of

\begin{namedtheorem}[\ref{thm:080}]
	Let $ m\in\Z_{>2n} $ and $ \xi=x+iy\in\calH_n $. We define the function $ f_{1,m,\xi}:\calH_n\to\C $,
	\[ f_{1,m,\xi}(z)=C_{m,n}^{-1}\,\det\left(\frac1{2i}\left(z-\overline\xi\right)\right)^{-m}. \]
	Let $ \Gamma $ be a congruence subgroup of $ \Sp_{2n}(\Z) $. Then, the Poincar\'e series
	\[ \Delta_{\Gamma,m,\xi}=P_\Gamma f_{1,m,\xi}=\sum_{\gamma\in\Gamma} f_{1,m,\xi}\big|_m\gamma \]
	converges absolutely and uniformly on compact subsets of $ \calH_n $ and defines a Siegel cusp form $ \Delta_{\Gamma,m,\xi}\in S_m(\Gamma) $ that satisfies
	\[ \scal f{\Delta_{\Gamma,m,\xi}}_{\Gamma\backslash\calH_n}=f(\xi),\qquad f\in S_m(\Gamma).  \]
\end{namedtheorem}

\bigskip

\noindent\emph{Acknowledgements.} I would like to thank Goran Mui\'c for turning my attention to the problems considered in this paper.

\section{Basic notation}

Throughout the paper, we fix $ m,n\in\Z_{>0} $. We will write $ z^\top $ for the transpose, $ z^* $ for the conjugate transpose, and $ \tr(z) $ for the trace of a matrix $ z\in M_n(\C) $. For every set $ S\subseteq\mathbb C $, let $ \Diag_n(S) $ be the set of diagonal matrices in $ M_n(\mathbb C) $ whose diagonal entries belong to $ S $.

Let $ J_n=\begin{pmatrix}&I_n\\-I_n&\end{pmatrix}\in\GL_{2n}(\C) $. The group 
\[ \Sp_{2n}(\R)=\left\{g\in\GL_{2n}(\R):g^\top J_ng=J_n\right\} \] 
acts on the left on the Siegel upper half-space
\[ \calH_n=\left\{z=x+iy\in M_n(\C):z^\top=z,\ y>0\right\} \]
transitively by generalized linear fractional transformations:
\[ g.z=(Az+B)(Cz+D)^{-1},\qquad g=\begin{pmatrix}A&B\\C&D\end{pmatrix}\in\Sp_{2n}(\R),\ z\in\calH_n. \]
Moreover, $ \Sp_{2n}(\R) $ acts on the right on the set $ \mathbb C^{\mathcal H_n} $ of functions $ \mathcal H_n\to\C $ by
\begin{equation}\label{eq:074}
	\left(f\big|_mg\right)(z)=j(g,z)^{-m}f(g.z),\qquad f\in\mathbb C^{\mathcal H_n},\ g\in\Sp_{2n}(\R), z\in\calH_n,
\end{equation}
where
\[ j(g,z)=\det(Cz+D),\qquad g=\begin{pmatrix}A&B\\C&D\end{pmatrix}\in\Sp_{2n}(\R),\ z\in\calH_n. \]
Recall that $ j $ is an automorphic factor: for all $ g_1,g_2\in\Sp_{2n}(\R) $ and $ z\in\calH_n $, we have
\begin{equation}\label{eq:040}
	j(g_1g_2,z)=j(g_1,g_2.z)\,j(g_2,z).
\end{equation}

Every $ g\in \Sp_{2n}(\R) $ has a unique factorization
\begin{equation}\label{eq:007}
	g=\begin{pmatrix}I_n&x\\&I_n\end{pmatrix}\begin{pmatrix}y^{\frac12}\\&y^{-\frac12}\end{pmatrix}\begin{pmatrix}A&B\\-B&A\end{pmatrix},
\end{equation}
where $ x,y,A,B\in M_n(\R) $ are such that $ x+iy=g.(iI_n) $ and $ A+iB\in \mathrm U(n) $. Let us denote the three factors on the right-hand side of \eqref{eq:007}, from left to right, by $ n_x $, $ a_y $, and $ k_{A+iB} $.
The assignment $ u\mapsto k_u $ defines a Lie group isomorphism from $ \mathrm U(n) $ to the maximal compact subgroup $ K=\Sp_{2n}(\R)\cap \mathrm U(2n) $
of $ \Sp_{2n}(\R) $. For every $ r\in\mathbb Z $, we define a character $ \chi_r:K\to\mathbb C^\times $,
\[ \chi_r(k_u)=\det u^r,\qquad u\in \mathrm U(n). \]

Let $ \mathsf v $ be the $ \Sp_{2n}(\R) $-invariant Radon measure on $ \calH_n $ given by \[ d\mathsf v(z)=d\mathsf v(x+iy)=(\det y)^{-n-1}\,\prod_{1\leq r\leq s\leq n}dx_{r,s}\,dy_{r,s}. \]
The generalized linear fractional transformation corresponding to the matrix $ \ell=\frac1{2i}\begin{pmatrix}I_n&-iI_n\\I_n&iI_n\end{pmatrix} $ defines a biholomorphic function from $ \calH_n $ to the bounded domain
\[ \calD_n=\left\{w\in M_n(\C):w^\top=w,\ I_n-w^*w>0\right\}. \]
We have
\[ \int_{\calH_n}f(z)\,d\mathsf v(z)=\int_{\calD_n}f(\ell^{-1}.w)\,d\mathsf v_\calD(w),\qquad f\in C_c(\calH_n), \]
where $ \mathsf v_\calD $ is the Radon measure on $ \calD_n $ given by 
\begin{equation}\label{eq:009}
d\mathsf v_\calD(w)=d\mathsf v_\calD(u+iv)=2^{n(n+1)}\det(I_n-w^*w)^{-n-1}\,\prod_{1\leq r\leq s\leq n}du_{r,s}dv_{r,s},
\end{equation}
and $ C_c(\mathcal H_n) $ is the space of compactly supported continuous functions $ \mathcal H_n\to\mathbb C $. 

It follows from \cite[Thm.\ 8.36]{knapp1996} that a Haar measure on $ \Sp_{2n}(\R) $ is given by
\begin{equation}\label{eq:059}
	\int_{\Sp_{2n}(\R)}f(g)\,dg=\int_{\calH_n}\int_Kf(n_xa_yk)\,dk\,d\mathsf v(z),\qquad f\in C_c(\Sp_{2n}(\R)), 
\end{equation}
where $ dk $ denotes the Haar measure on $ K $ with respect to which $ \vol(K)=1 $. 
For every discrete subgroup $ \Gamma $ of $ \Sp_{2n}(\R) $, there is a unique Radon measure on $ \Gamma\backslash\Sp_{2n}(\R) $ such that
\begin{equation}\label{eq:088}
	\int_{\Gamma\backslash\Sp_{2n}(\R)}\sum_{\gamma\in\Gamma}f(\gamma g)\,dg=\int_{\Sp_{2n}(\R)}f(g)\,dg,\qquad f\in C_c(\Sp_{2n}(\R)),
\end{equation}
or equivalently
\begin{equation}\label{eq:016}
\int_{\Gamma\backslash\Sp_{2n}(\R)}f(g)\,dg=\varepsilon_\Gamma^{-1}\,\int_{\Gamma\backslash\calH_n}\int_Kf(n_xa_yk)\,dk\,d\mathsf v(z),\qquad f\in C_c(\Gamma\backslash\Sp_{2n}(\R)), 
\end{equation}
where $ \varepsilon_\Gamma=\abs{\Gamma\cap\left\{\pm I_{2n}\right\}} $.
Using this measure, we define in a standard way the Hilbert space $ L^2(\Gamma\backslash\Sp_{2n}(\R)) $ of (equivalence classes of) square-integrable functions $ \Gamma\backslash\Sp_{2n}(\R)\to\C $. 

By \cite[Ch.\ 1, Lem.\ 2.8]{andrianov_zhuravlev}, we have
\begin{equation}\label{eq:077}
	\Im(g.z)=(Cz+D)^{-*}y(Cz+D)^{-1},\quad g=\begin{pmatrix}A&B\\C&D\end{pmatrix}\in\Sp_{2n}(\R),\ z=x+iy\in\calH_n,
\end{equation}
which implies that for all $ f:\calH_n\to\C $ and $ g\in\Sp_{2n}(\R) $, we have
\begin{equation}\label{eq:065}
	\sup_{z\in\calH_n}\abs{\left(f\big|_mg\right)(z)}\,\det y^{\frac m2}=\sup_{z\in\calH_n}\abs{f(z)}\,\det y^{\frac m2}
\end{equation}
and if $ f $ is Borel measurable, then
\begin{equation}\label{eq:078}
	\int_{\calH_n}\abs{\left(f\big|_mg\right)(z)}\,\det y^{\frac m2}\,d\mathsf v(z)=\int_{\calH_n}\abs{f(z)}\,\det y^{\frac m2}\,d\mathsf v(z).
\end{equation}

Given a real Lie algebra $ \mathfrak g $, we will denote by $ \mathfrak g_\C $ its complexification $ \mathfrak g\otimes_\R\C $.
For $ r=1,\ldots,n $, let $ T_r=-i\begin{pmatrix}&E_{r,r}\\-E_{r,r}&\end{pmatrix}\in\mathfrak{gl}_{2n}(\C) $, where $ E_{r,r}\in M_n(\C) $ is the matrix unit with $ (r,r) $-th entry $ 1 $. The Lie algebra $ \mathfrak h=i\bigoplus_{r=1}^n\mathbb R T_r $ is a Cartan subalgebra both of $ \mathfrak g=\Lie(\Sp_{2n}(\R))\equiv\mathfrak{sp}_{2n}(\R) $ and of $ \mathfrak k=\Lie(K) $. Let $ e_r $ denote the linear functional on $ \mathfrak h_\C $ such that $ e_r(T_s)=\delta_{r,s} $ for all $ s $. We fix the following choice of the set $ \Delta_K^+ $ of compact positive roots and the set $ \Delta_n^+ $ of  non-compact positive roots in $ \Delta(\mathfrak h_\C:\mathfrak g_\C) $:
\[ \begin{aligned}
\Delta_K^+&=\left\{e_r-e_s:1\leq r<s\leq n\right\},\\
\Delta_n^+&=\left\{e_r+e_s:1\leq r\leq s\leq n\right\}.
\end{aligned} \]
Let us denote by $ \mathfrak g_\alpha $ the root subspace of $ \mathfrak g_\C $ corresponding to a root $ \alpha $. We define
\[ \mathfrak k_\C^+=\bigoplus_{\alpha\in\Delta_K^+}\mathfrak g_{\alpha},\quad\mathfrak k_\C^-=\bigoplus_{\alpha\in\Delta_K^+}\mathfrak g_{-\alpha},\quad\mathfrak p_\C^+=\bigoplus_{\alpha\in\Delta_n^+}\mathfrak g_{\alpha},\quad\text{and}\quad\mathfrak p_\C^{-}=\bigoplus_{\alpha\in\Delta_n^+}\mathfrak g_{-\alpha}. \]

Let $ \calU(\frakg) $ be the universal enveloping algebra of $ \mathfrak g_\C $, and let $ \calZ(\frakg) $ denote the center of $ \calU(\frakg) $.
Given a $ \mathfrak g $-module $ V $ and a subset $ U $ of $ \mathcal U(\mathfrak g) $, let $ V^U $ denote the subspace of vectors  $ v\in V $ such that $ u.v=0 $  for all $u\in U $.
A vector $ v\in V $ is said to be
 $ \mathcal Z(\mathfrak g) $-finite if $ \dim_{\C}\mathcal Z(\mathfrak g).v<\infty $. 

Given a unitary representation $ (\pi,H) $ of $ \Sp_{2n}(\R) $, let $ H_K $ denote the $ (\mathfrak g,K) $-module of smooth vectors $ v\in H $ that are $ K $-finite, i.e., such that $ \dim_{\C}\mathrm{span}_{\mathbb C}\pi(K)v<\infty $. Moreover, for every irreducible unitary representation $ \rho $ of $ \Sp_{2n}(\R) $, let us denote by $ [\rho] $ the unitary equivalence class of $ \rho $, and by $ H_{[\rho]} $ the sum of irreducible closed $ \Sp_{2n}(\R) $-invariant subspaces of $ H $ that are equivalent to $ \rho $.
Similarly, given a $ K $-module $ V $ and an irreducible representation $ \tau $ of $ K $, we denote by $ [\tau] $ the equivalence class of $ \tau $, and by $ V_{[\tau]} $ the sum of $ K $-invariant subspaces of $ V $ that are equivalent to $ \tau $.

\section{Preliminaries on certain highest weight representations}\label{sec:003}

Throguhout this section, suppose that $ m\in\mathbb Z_{>n} $. Let $ H_m $ be the Hilbert space of holomorphic functions $ f:\mathcal H_n\to\C $ such that
\[ \int_{\calH_n}\abs{f(z)}^2\det y^m\,d\mathsf v(z)<\infty, \]
with the inner product
\[ \scal{f_1}{f_2}_{H_m}=\int_{\calH_n}f_1(z)\,\overline{f_2(z)}\,\det y^m\,d\mathsf v(z). \]
The irreducible unitary representation $ (\pi_m,H_m) $ of $ \mathrm{Sp}_{2n}(\R) $ defined by
\[ \pi_m(g)f=f\big|_mg^{-1},\qquad g\in\mathrm{Sp}_{2n}(\R),\ f\in H_m, \]
belongs to the holomorphic discrete series of $ \Sp_{2n}(\R) $ (see, e.g., \cite[\S3]{gelbart}).

Another useful realization $ (\pi_m^\ell,D_{m}) $ of $ \pi_m $ can be defined by requiring that the assignment $ f\mapsto f\big|_m\ell^{-1} $ define a unitary $ \Sp_{2n}(\R) $-equivalence $ H_m\to D_{m} $. One checks easily that $ D_{m} $ is the Hilbert space of holomorphic functions $ f:\mathcal D_n\to\mathbb C $ such that 
\[ \int_{\mathcal D_n}\abs{f(w)}^2\det(I_n-w^*w)^m\,d\mathsf v_{\mathcal D}(w)<\infty,  \]
with the inner product
\begin{equation}\label{eq:055}
	 \scal{f_1}{f_2}_{D_m}=\int_{\calD_n}f_1(w)\,\overline{f_2(w)}\,\det(I_n-w^*w)^m\,d\mathsf v_{\mathcal D}(w),
\end{equation}
and that
\[ \pi_m^{\ell}(g)f=f\big|_m\ell g^{-1}\ell^{-1},\qquad g\in\Sp_{2n}(\R),\ f\in D_m. \]
For every $ d\in\mathbb Z_{\geq0} $, let $ D_{m,d} $ be the finite-dimensional $ K $-invariant subspace of $ D_m $ consisting of the functions $ p_\mu:\mathcal D_n\to\C $,
\begin{equation}\label{eq:086}
p_\mu(w)=\mu(w),
\end{equation}
where $ \mu $ is a homogeneous polynomial of degree $ d $ in $ \C[X_{r,s}:1\leq r,s\leq n] $.
The $ (\mathfrak g,K) $-module $ (D_m)_K $ of $ K $-finite vectors in $ D_m $ has a direct sum decomposition
\begin{equation}\label{eq:028}
	 (D_m)_K=\bigoplus_{d\in\mathbb Z_{\geq0}}D_{m,d}. 
\end{equation}
Writing $ D_{m,-1}=0 $, one shows easily that $ \pi_m^{\ell}(\mathfrak p_\C^\pm)D_{m,d}\subseteq D_{m,d\mp1} $ for every $ d $. In particular, the $ 1 $-dimensional space $ D_{m,0}\subseteq(D_m)_{[\chi_{-m}]} $ is annihilated by $ \mathfrak p_\C^+ $. In fact:

\begin{lemma}\label{lem:019}
	Let $ m\in\mathbb Z_{>n} $. Then:
	\begin{enumerate}[label=\textup{(\roman*)}]
		\item The representation $ (\pi_m,H_m) $ is, up to unitary equivalence, the unique irreducible unitary representation  of $ \Sp_{2n}(\R) $ containing a one-dimensional subspace that is $ K $-equivalent to $ \chi_{-m} $ and annihilated by $ \mathfrak p_\C^+ $. 
		\item\label{lem:019:2} The contragredient representation $ (\pi_m^*,H_m^*) $ is, up to unitary equivalence, the unique irreducible unitary representation of $ \Sp_{2n}(\R) $ containing a one-dimensional subspace that is $ K $-equivalent to $ \chi_m $ and annihilated by $ \mathfrak p_\C^- $. 
	\end{enumerate}
\end{lemma}

\begin{proof}
	This result follows from the theory of Verma modules. More precisely, by \cite[first part of Thm.\ 3.4.11]{wallachI} and \cite[\S20.3, Thm.\ A]{humphreys}, the stated properties uniquely determine the $ \mathfrak g $-module isomorphism class or equivalently, by the connectedness of $ K $ and \cite[Cor.\ 3.49]{hall}, the $ \left(\mathfrak g,K\right) $-module isomorphism class of $ (H_m)_K $ (resp., $ (H_m^*)_K $). Thus, by Harish-Chandra's result \cite[second part of Thm.\ 3.4.11]{wallachI} they uniquely determine the unitary equivalence class of the representation $ \pi_m $ (resp., $ \pi_m^* $).  
\end{proof}

In terms of \cite[Thm.\ 9.20]{knapp1986}, $ \pi_m $ is a discrete series representation of $ \Sp_{2n}(\R) $ with Harish-Chandra parameter $ -\sum_{r=1}^n(m-r)e_r $, Blattner parameter $ -m\sum_{r=1}^ne_r $, and highest $ K $-type $ \chi_{-m} $. The contragredient representation $ \pi_m^* $ is an antiholomorphic discrete series representation of $ \Sp_{2n}(\R) $ with Harish-Chandra parameter $ \sum_{r=1}^n(m-r)e_r $, Blattner parameter $ m\sum_{r=1}^ne_r $, and lowest $ K $-type $ \chi_{m} $. In particular,
\begin{equation}\label{eq:084}
	\dim_\C(H_m)_{[\chi_{-m}]}=1\qquad\text{and}\qquad\dim_\C(H_m^*)_{[\chi_{m}]}=1.
\end{equation}

Recall that an irreducible unitary representation $ (\pi,H) $ of $ \Sp_{2n}(\R) $ is said to be integrable if for all  $ h,h'\in H_K $, the matrix coefficient $ c_{h,h'}:\Sp_{2n}(\R)\to\C $,
\[ c_{h,h'}(g)=\scal{\pi(g)h}{h'}_H, \] 
belongs to $ L^1(\Sp_{2n}(\R)) $. By \cite{hecht_schmid}, we have the following

\begin{lemma}\label{lem:058}
	Let $ m\in\Z_{>n} $. The representation $ \pi_m $ (resp., $ \pi_m^* $) is integrable if and only if $ m>2n $.
\end{lemma}

\section{Siegel cusp forms}\label{sec:021}

Throughout this section, let $ \Gamma $ be a congruence subgroup of $ \Sp_{2n}(\Z) $ of level $ N\in\mathbb Z_{>0} $, i.e., a subgroup of $ \Sp_{2n}(\Z) $ containing the principal level $ N $ congruence subgroup
\[ \Gamma_n(N)=\left\{g\in\Sp_{2n}(\Z):g\equiv I_{2n}\pmod N\right\}. \]
A Siegel cusp form of weight $ m $ for $ \Gamma $ is a holomorphic function $ f:\calH_n\to\C $ with the following two properties:
\begin{enumerate}[label=\textup{(S\arabic*)}]
	\item\label{enum:002:1} $ f\big|_m\gamma=f $ for all $ \gamma\in\Gamma $.
	\item\label{enum:002:2} $ \sup_{z\in\calH_n}\abs{f(z)}\det y^{\frac m2}<\infty $.
\end{enumerate}
Such functions constitute a finite-dimensional complex vector space $ S_m(\Gamma) $, which we equip with the Petersson inner product
\begin{equation}\label{eq:017}
\scal{f_1}{f_2}_{\Gamma\backslash\calH_n}=\varepsilon_\Gamma^{-1}\int_{\Gamma\backslash\calH_n}f_1(z)\,\overline{f_2(z)}\,\det y^m\,d\mathsf v(z),
\end{equation}
where $ \varepsilon_\Gamma=\abs{\Gamma\cap\left\{\pm I_{2n}\right\}} $.

Let $ \mathcal A_{cusp}(\Gamma\backslash\Sp_{2n}(\R)) $ be the space of cuspidal automorphic forms for $ \Gamma $, i.e., smooth functions $ \varphi:\Sp_{2n}(\R)\to\C $ with the following properties:
\begin{enumerate}[label=\textup{(A\arabic*)}]
	\item\label{enum:001:1} $ \varphi(\gamma g)=\varphi(g) $ for all $ \gamma\in\Gamma $ and $ g\in\Sp_{2n}(\R) $.
	\item $ \varphi $ is $ \mathcal Z(\mathfrak g) $-finite, i.e., $ \dim_\C\calZ(\mathfrak g)\varphi<\infty $.
	\item $ \varphi $ is $ K $-finite on the right, i.e., $ \dim_\C\mathrm{span}_\C\left\{\varphi(\spacedcdot k):k\in K\right\}<\infty $.
	\item\label{enum:001:4} $ \varphi $ is cuspidal, i.e., for every proper $ \Q $-parabolic subgroup $ P $ of $ \Sp_{2n} $, we have
	\[ \int_{(\Gamma\cap U(\R))\backslash U(\R)}\varphi(ug)\,du=0,\qquad g\in\Sp_{2n}(\R), \]
	where $ U $ is the unipotent radical of $ P $.	
	\item\label{enum:001:5} $ \varphi $ satisifies the following, assuming \ref{enum:001:1}--\ref{enum:001:4} mutually equivalent \cite[\S1.8]{bj}, conditions:
	\begin{enumerate}
		\item $ \varphi $ is bounded.
		\item $ \varphi\in L^2(\Gamma\backslash\Sp_{2n}(\R)) $.
	\end{enumerate}
\end{enumerate}
The space $ \mathcal A_{cusp}(\Gamma\backslash\Sp_{2n}(\R)) $ is the $ (\mathfrak g,K) $-module of $ \mathcal Z(\mathfrak g) $-finite, $ K $-finite, smooth vectors in the right regular representation of $ \Sp_{2n}(\R) $ on the space $ L^2_{cusp}(\Gamma\backslash\Sp_{2n}(\R)) $ of  cuspidal (equivalence classes of) functions in $ L^2(\Gamma\backslash\Sp_{2n}(\R)) $ \cite[\S2.2]{bj}. 

Next, let us introduce the classical lift $ \Phi_m $ that maps each function $ f:\mathcal H_n\to\mathbb C $ to the function $ F_f:\Sp_{2n}(\R)\to\C $,
\begin{equation}\label{eq:013}
	F_f(g)=\left(f\big|_mg\right)(iI_n)=\chi_m(k_u)\det y^{\frac m2}f(x+iy),\quad g=n_xa_yk_u\in\Sp_{2n}(\R).
\end{equation}
One checks easily that for every $ \gamma\in\Sp_{2n}(\R) $,
\begin{equation}\label{eq:014}
f\big|_m\gamma=f \qquad\text{if and only if}\qquad  F_f(\gamma\spacedcdot)=F_f.
\end{equation}
Moreover, by an obvious variant of \cite[Lem.\ 7]{asgari}, for every smooth function $ f:\calH_n\to\C $ the following equivalence holds:
\begin{equation}\label{eq:083}
f\text{ is holomorphic}\quad\Leftrightarrow\quad \mathfrak p_\C^-F_f=0.  
\end{equation}

Lacking a suitable reference, we sketch a proof of the following well-known

\begin{proposition}\label{prop:034}
	The assignment $ f\mapsto F_f $ defines a unitary isomorphism 
	\[ \Lambda_m:S_m(\Gamma)\to\mathcal A_{cusp}(\Gamma\backslash\mathrm{Sp}_{2n}(\R))_{[\chi_{m}]}^{\mathfrak p_\C^-}, \]
	where $ \mathcal A_{cusp}(\Gamma\backslash\mathrm{Sp}_{2n}(\R))_{[\chi_{m}]}^{\mathfrak p_\C^-} $ is regarded as a subspace of the Hilbert space $ L^2(\Gamma\backslash\Sp_{2n}(\R)) $.
\end{proposition}

\begin{proof}
	Let $ f\in S_m(\Gamma) $. The function $ F_f $ is obviously smooth and satisfies \ref{enum:001:1} by \eqref{eq:014}. It transforms on the right as $ \chi_m $ by \eqref{eq:013} and is hence $ K $-finite on the right and annihilated by $ \mathfrak k_\C^+ $. It is also annihilated by $ \mathfrak p_\C^- $ by \eqref{eq:083} and is hence $ \mathcal Z(\mathfrak g) $-finite, as  $ \mathcal Z(\mathfrak g)\subseteq\calU(\mathfrak k)+\calU(\mathfrak g)(\mathfrak k_\C^++\mathfrak p_\C^-) $ by an application of \cite[Lem.\ 8.17]{knapp1986} with a suitable choice of positive roots. Moreover, $ F_f $ is cuspidal by an obvious modification of \cite[proof of Lem.\ 5]{asgari}, and we have $ \norm{F_f}_{L^2(\Gamma\backslash\Sp_{2n}(\R))}=\norm f_{S_m(\Gamma)} $ by \eqref{eq:016}, \eqref{eq:013}, and \eqref{eq:017}. Thus, $ F_f\in\mathcal A_{cusp}(\Gamma\backslash\mathrm{Sp}_{2n}(\R))_{[\chi_{m}]}^{\mathfrak p_\C^-} $, and $ \Lambda_m $ is a well-defined linear isometry.
	
	To prove the surjectivity of $ \Lambda_m $, let $ \varphi\in\mathcal A_{cusp}(\Gamma\backslash\mathrm{Sp}_{2n}(\R))_{[\chi_{m}]}^{\mathfrak p_\C^-} $. The unique function $ f:\mathcal H_n\to\C $ such that $  \varphi=F_f $ is obviously given by
	$ f(z)=\varphi(n_xa_y)\,\det y^{-\frac m2} $ for $ z=x+iy\in\calH_n $.
	This function is holomorphic by \eqref{eq:083}, satisfies \ref{enum:002:1} by \eqref{eq:014}, and satisfies \ref{enum:002:2} by the boundedness of $ \varphi $, hence belongs to $ S_m(\Gamma) $. This proves the claim.
\end{proof}

In Prop.\ \ref{prop:081} below, we will give another description of the image $ \mathcal A_{cusp}(\Gamma\backslash\mathrm{Sp}_{2n}(\R))_{[\chi_{m}]}^{\mathfrak p_\C^-} $ of the unitary isomorphism $ \Lambda_m $. To this end, we note that lemma \cite[Lem.\ 77]{harish1966} continues to hold after replacing the space $ L^2(G) $ in its statement by $L^2(\Delta\backslash G)$, where $ \Delta $ is a discrete subgroup of $ G $. In particular, we have the following

\begin{lemma}\label{lem:018}
	Let $ \varphi\in\mathcal A_{cusp}(\Gamma\backslash\mathrm{Sp}_{2n}(\R))\setminus\left\{0\right\} $. Then, the smallest closed $ \Sp_{2n}(\R) $-invariant subspace of $ L^2_{cusp}(\Gamma\backslash\Sp_{2n}(\R)) $ containing $ \varphi $ is an orthogonal sum of finitely many irreducible closed $ \Sp_{2n}(\R) $-invariant subspaces. 
\end{lemma}

\begin{lemma}\label{lem:017}
	Suppose that $ m\in\Z_{>n} $. Let $ \varphi\in\mathcal A_{cusp}(\Gamma\backslash\mathrm{Sp}_{2n}(\R))_{[\chi_{m}]}^{\mathfrak p_\C^-}\setminus\left\{0\right\} $. Then, the smallest closed $ \Sp_{2n}(\R) $-invariant subspace $ E $ of $ L^2_{cusp}(\Gamma\backslash\Sp_{2n}(\R)) $ containing $ \varphi $ is unitarily equivalent to $ \pi_m^* $, and $ E_{[\chi_m]}=\C\varphi $.
	
\end{lemma}

\begin{proof}
	By Lem.\ \ref{lem:018}, $ E $ is an orthogonal sum of finitely many irreducible closed $ \Sp_{2n}(\R) $-invariant subspaces $ E_r $, $ r=1,\ldots,p $. In particular, $ E $ is admissible by \cite[Thm.\ 3.4.10]{wallachI}, hence  by Harish-Chandra's result \cite[Thm.\ 0.4]{knapp_vogan} and the definition of $ E $, $ \varphi $ generates the $ (\mathfrak g,K) $-module $ E_K $.
	
	By \eqref{eq:084}, we can fix $ f_m^*\in H_m^*\setminus\left\{0\right\} $ such that $ (H_m^*)_{[\chi_m]}=\C f_m^* $. For each $ r $, the orthogonal projection $ \varphi_r $ of $ \varphi $ onto $ E_r $ is non-zero by the definition of $ E $, transforms on the right as $ \chi_m $, and is annihilated by $ \mathfrak p_\C^- $, hence $ E_r\cong\left(\pi_m^*,H_m^*\right) $ by Lem.\ \ref{lem:019}\ref{lem:019:2}, and $ (E_r)_{[\chi_m]}=\C\varphi_r $ by \eqref{eq:084}. Thus, by Schur's lemma \cite[Lem.\ 1.2.1]{wallachI}, there exists a unique continuous $ \Sp_{2n}(\R) $-equivariant operator $ P_r:H_m^*\to E_r $ such that $ P_rf_m^*=\varphi_r $. The operator $ P=\sum_{r=1}^pP_r:H_m^*\to E $ is continuous and $ \Sp_{2n}(\R) $-equivariant, and $ P f_m^*=\varphi $. Since the $ (\mathfrak g,K) $-module $ (H_m^*)_K $ is irreducible by \cite[first part of Thm.\ 3.4.11]{wallachI}, and $ E_K $ is generated by $ \varphi $, $ P $ restricts to a $ (\mathfrak g,K) $-module isomorphism $ (H_m^*)_K\to E_K $. Thus, $ E $ is infinitesimally equivalent to the irreducible unitary representation $ \pi_m^* $, hence by \cite[Thm.\ 3.4.12 and second part of Thm.\ 3.4.11]{wallachI} it is unitarily equivalent to $ \pi_m^* $. In particular, the space $ E_{[\chi_m]} $ is $ 1 $-dimensional and hence, as it obviously contains $ \varphi $, equals $ \C\varphi $.
\end{proof}

\begin{proposition}\label{prop:081}
	Suppose that $ m\in\Z_{>n} $. Then,
	\begin{equation}\label{eq:033}
	\mathcal A_{cusp}(\Gamma\backslash\mathrm{Sp}_{2n}(\R))_{[\chi_{m}]}^{\mathfrak p_\C^-}=\left(L^2_{cusp}(\Gamma\backslash\Sp_{2n}(\R))_{[\pi_m^*]}\right)_{[\chi_m]}.
	\end{equation}
\end{proposition}

\begin{proof}
	The inclusion from left to right holds by Lem.\ \ref{lem:017}. The reverse inclusion follows easily from Lem.\ \ref{lem:019}\ref{lem:019:2} and \eqref{eq:084} once we remember the well-known fact that for every irreducible unitary representation $ \pi $ of $ \Sp_{2n}(\R) $, the space $ L^2_{cusp}(\Gamma\backslash\Sp_{2n}(\R))_{[\pi]} $ is an orthogonal sum of finitely many (possibly zero) irreducible closed $ \Sp_{2n}(\R) $-invariant subspaces.
\end{proof}

\section{Matrix coefficients and Poincar\'e series}\label{sec:039}

Let $ \Gamma $ be a discrete subgroup of $ \Sp_{2n}(\R) $. For every $ \varphi\in L^1(\Sp_{2n}(\R)) $, the Poincar\'e series
\[ (P_\Gamma\varphi)(g)=\sum_{\gamma\in\Gamma}\varphi(\gamma g) \]
converges absolutely almost everywhere on $ \Sp_{2n}(\R) $, and $ P_\Gamma\varphi\in L^1(\Gamma\backslash\Sp_{2n}(\R)) $ \cite[\S4]{muic09}. This in particular holds if $ \varphi $ is a $ K $-finite matrix coefficient $ c_{h,h'} $ of an integrable representation $ \pi $ of $ \Sp_{2n}(\R) $. Moreover, since such a matrix coefficient is obviously a $ \mathcal Z(\mathfrak g) $-finite and left and right $ K $-finite smooth function on $ \Sp_{2n}(\R) $, the Poincar\'e series $ P_\Gamma c_{h,h'} $ converges absolutely and uniformly on compact subsets of $ \Sp_{2n}(\R) $ and defines a bounded function on $ \Sp_{2n}(\R) $ by \cite[proof of Thm.\ 3.10(i)]{muic09} and \cite[Lem.\ 6.3]{muic19}.

From now on, let $ \Gamma $ be a congruence subgroup of $ \Sp_{2n}(\Z) $. Since $ \vol(\Gamma\backslash\Sp_{2n}(\R))<\infty $, the boundedness of $ P_\Gamma c_{h,h'} $ and the cuspidality of $ c_{h,h'} $ \cite[7.7.1]{wallachI} imply that $ P_\Gamma c_{h,h'}\in L^2_{cusp}(\Gamma\backslash\Sp_{2n}(\R)) $. In fact, we have the following

\begin{lemma}\label{lem:023}
	 Let $ (\pi,H) $ be an integrable representation of $ \Sp_{2n}(\R) $. Then,
	\begin{equation}\label{eq:022}
		 \left(L^2_{cusp}(\Gamma\backslash\Sp_{2n}(\R))_{[\pi]}\right)_K=\mathrm{span}_\C\left\{P_\Gamma c_{h,h'}:h,h'\in H_K\right\}. 
	\end{equation}
\end{lemma}

\begin{proof}
	This is a special case of Mili\v ci\'c's result \cite[Lem.\ 6.6]{muic19}. More precisely, by \cite[Lem.\ 6.6]{muic19} we have
	\begin{equation}\label{eq:020}
		L^2_{cusp}(\Gamma\backslash\Sp_{2n}(\R))_{[\pi]}=\mathrm{Cl}_{L^2(\Gamma\backslash\Sp_{2n}(\R))}\sum_{h'\in H_K}P_\Gamma\mathcal B_{h'},
	\end{equation}
	where $ \mathcal B_{h'}=\mathrm{Cl}_{L^1(\Sp_{2n}(\R))}\left\{c_{h,h'}:h\in H_K\right\} $. Since by \cite[Thm.\ 6.4(i)]{muic19} the operator $ P_\Gamma:\mathcal B_{h'}\to L^2(\Gamma\backslash\Sp_{2n}(\R)) $ is continuous, the $ (\mathfrak g,K) $-module $ P_\Gamma\left(\left\{c_{h,h'}:h\in H_K\right\}\right) $ is dense in $ P_\Gamma\mathcal B_{h'} $, hence by \eqref{eq:020} the $ (\mathfrak g,K) $-module 
	\[ \mathrm{span}_\C\left\{P_\Gamma c_{h,h'}:h,h'\in H_K\right\}=\sum_{h'\in H_K}P_\Gamma\left(\left\{c_{h,h'}:h\in H_K\right\}\right) \]
	is dense in $ L^2_{cusp}(\Gamma\backslash\Sp_{2n}(\R))_{[\pi]} $. This implies \eqref{eq:022} by Harish-Chandra's result \cite[Thm.\ 0.4]{knapp_vogan} since $ L^2_{cusp}(\Gamma\backslash\Sp_{2n}(\R))_{[\pi]} $ is an admissible representation of $ \Sp_{2n}(\R) $ by the proof of Prop.\ \ref{prop:081}.
\end{proof}
 
\begin{proposition}\label{prop:029}
	Suppose that $ m\in\Z_{>2n} $. Then, 
	\begin{equation}\label{eq:024}
		\left(L^2_{cusp}(\Gamma\backslash\Sp_{2n}(\R))_{[\pi_m^*]}\right)_{[\chi_m]}=\left\{P_\Gamma c_{h,h'}:h\in \left(H_m^*\right)_{[\chi_m]},\ h'\in \left(H_m^*\right)_K\right\}.
	\end{equation}
\end{proposition} 
 
\begin{proof}
	By \cite[Prop.\ 1.18(a)]{knapp_vogan}, we have a direct sum decomposition
	\begin{equation}\label{eq:025}
	\left(L^2_{cusp}(\Gamma\backslash\Sp_{2n}(\R))_{[\pi_m^*]}\right)_K=\bigoplus_{[\tau]\in\hat K}\left(L^2_{cusp}(\Gamma\backslash\Sp_{2n}(\R))_{[\pi_m^*]}\right)_{[\tau]}, 
	\end{equation}
	where $ \hat K $ denotes the unitary dual of $ K $. On the other hand, by Lem.\ \ref{lem:023}, we have
	\begin{equation}\label{eq:026}
	 \left(L^2_{cusp}(\Gamma\backslash\Sp_{2n}(\R))_{[\pi_m^*]}\right)_K=\sum_{[\tau]\in\widehat K}\mathrm{span}_\C\left\{P_\Gamma c_{h,h'}:h\in(H_m^*)_{[\tau]},h'\in(H_m^*)_K\right\}.  
	\end{equation}
	Since for every $ \tau\in\widehat K $ obviously
	\begin{equation}\label{eq:027}
	\mathrm{span}_\C\left\{P_\Gamma c_{h,h'}:h\in (H_m^*)_{[\tau]},h'\in (H_m^*)_K\right\}\subseteq \left(L^2_{cusp}(\Gamma\backslash\Sp_{2n}(\R))_{[\pi_m^*]}\right)_{[\tau]}, 
	\end{equation}
	\eqref{eq:025} and \eqref{eq:026} imply that in fact the equality holds in $ \eqref{eq:027} $. In particular, we have \eqref{eq:024}.
\end{proof} 

Suppose that $ m\in\Z_{>n} $.
 By \eqref{eq:028}, 
\begin{equation}\label{eq:090}
(H_m)_K=\left\{f_{\mu,m}:\mu\in\C[X_{r,s}:1\leq r,s\leq n]\right\}, 
\end{equation}
where the functions $ f_{\mu,m}:\calH_n\to\C $ are defined by
\begin{equation}\label{eq:041}
	f_{\mu,m}(z)=\left(p_\mu\big|_m\ell\right)(z)=(2i)^{mn}\,\frac{\mu((z-iI_n)(z+iI_n)^{-1})}{\det(z+iI_n)^m},\qquad z\in\calH_n. 
\end{equation}
Thus, writing $ f^*=\scal\spacedcdot{f}_{H_m}\in H_m^* $ for $ f\in H_m $, we have
\begin{equation}\label{eq:038}
(H_m^*)_K=\left\{f_{\mu,m}^*:\mu\in\C[X_{r,s}:1\leq r,s\leq n]\right\}
\end{equation}
and
\begin{equation}\label{eq:037}
(H_m^*)_{[\chi_m]}=\C f_{1,m}^*.
\end{equation}
Let us define a positive constant
\begin{equation}\label{eq:085}
	\begin{aligned}C_{m,n}&=\norm{F_{f_{1,m}}}_{L^2(\Sp_{2n}(\R))}^2\underset{\eqref{eq:013}}{\overset{\eqref{eq:059}}=}\norm{f_{1,m}}_{H_m}^2=\norm{p_1}_{D_m}^2\\&
	\underset{\eqref{eq:009}}{\overset{\eqref{eq:055}}=}2^{n(n+1)}\int_{\calD_n}\det(I_n-w^*w)^{m-n-1}\,\prod_{1\leq r\leq s\leq n}du_{r,s}dv_{r,s}\\
	&=2^{\frac{n(n+3)}2}\pi^{\frac{n(n+1)}2}\prod_{r=1}^n\frac{\Gamma\left(m-\frac{n+r}2\right)}{\Gamma\left(m-\frac{r-1}2\right)},
\end{aligned} 
\end{equation}
where the last equality holds by \cite[(2.3.1)]{hua63} (cf.\ \cite[III, \S6, Rem.\ after Prop.\ 1]{klingen}). In the last equality, $ \Gamma $ denotes the gamma function: $ \Gamma(s)=\int_0^\infty t^{s-1}e^{-t}\,dt $, $ \Re(s)>0 $.  
The second part of the following proposition gives an explicit formula for the matrix coefficients $ c_{h,h'} $ of Prop.\ \ref{prop:029}. 

\begin{proposition}\label{prop:057}
	Suppose that $ m\in\Z_{>n} $. Let $ f\in (H_m)_K $. Then:
	\begin{enumerate}[label=\textup{(\roman*)}]
		\item The matrix coefficient $ c_{f,f_{1,m}} $ of $ (\pi_m,H_m) $ is given by
		\begin{equation}\label{eq:048}
			c_{f,f_{1,m}}=C_{m,n}\,\check F_f,
		\end{equation}
		where we use the notation $ \check\varphi=\varphi(\spacedcdot^{-1}) $ for $ \varphi:\Sp_{2n}(\R)\to\C $.
		\item\label{prop:057:2} The matrix coefficient $ c_{f_{1,m}^*,f^*} $ of $ (\pi_m^*,H_m^*) $ is given by
		\begin{equation}\label{eq:032}
		c_{f_{1,m}^*,f^*}=C_{m,n}\,F_f. 
		\end{equation}
	\end{enumerate}
\end{proposition}

\begin{proof}
	The following proof is a generalization of \cite[proof of Lem.\ 3-5]{muic10}.
	One checks easily that the assignment $ f\mapsto F_f $ defines a unitary $ \Sp_{2n}(\R) $-equivalence $ \Psi_m $ from $ (\pi_m,H_m) $ onto a closed $ \Sp_{2n}(\R) $-invariant subspace $ L_m $ of the left regular representation $ (L,L^2(\Sp_{2n}(\R))) $. In particular, for a fixed $ f\in(H_m)_K\setminus\left\{0\right\} $, the function $ F_f\in(L_m)_K $ is $ \mathcal Z(\mathfrak g) $-finite and $ K $-finite on the left, and by \eqref{eq:013} it is also $ K $-finite on the right. Thus, by \cite[Thm.\ 1]{harish1966} there exists a function $ \alpha\in C_c^\infty(\Sp_{2n}(\R)) $ such that $ F_f=F_f*\check\alpha $. Denoting by $ \mathrm{pr}_{L_m}\alpha $ the orthogonal projection of $ \alpha $ onto $ L_m $, we have
	\[ \begin{aligned}
		\check F_f(g)&=(F_f*\check\alpha)(g^{-1})\\
		&=\int_{\Sp_{2n}(\R)}\alpha(h)\,F_f(g^{-1}h)\,dh\\
		&=\scal{L(g)F_f}{\alpha}_{L^2(\Sp_{2n}(\R))}\\
		&=\scal{L(g)F_f}{\mathrm{pr}_{L_m}\alpha}_{L_m},\qquad g\in\Sp_{2n}(\R).
	\end{aligned} \]
	Applying the inverse of $ \Psi_m $, it follows that
	\begin{equation}\label{eq:030}
		\check F_f=c_{f,h} 
	\end{equation}
	for some $ h\in H_m $. It is well-known that the assignment $ (h')^*\mapsto \check{c}_{(h')^*,f^*}= c_{f,h'} $ defines an $ \Sp_{2n}(\R) $-equivalence from $ H_m^* $ onto a closed $ \Sp_{2n}(\R) $-invariant subspace of $ (L,L^2(\Sp_{2n}(\R))) $ (see, e.g., \cite[proof of Prop.\ 1.3.3]{wallachI}). As $ \check F_f $ transforms on the left as $ \chi_m $ by \eqref{eq:013}, it follows from \eqref{eq:030} that $ h^*\in (H_m^*)_{[\chi_{m}]} $, hence $ h\in \C f_{1,m} $ by \eqref{eq:037}. Thus, there exists a function $ \lambda:(H_m)_K\to\C $ such that
	\begin{equation}\label{eq:031}
		c_{f,f_{1,m}}=\lambda(f)\check F_f,\qquad f\in(H_m)_K. 
	\end{equation}
	But the composition $ f\mapsto c_{f,f_{1,m}}=\check F_{\lambda(f)f}\mapsto\lambda(f)f $ is a linear operator on $ (H_m)_K $, hence the function $ \lambda $ is constant on $ (H_m)_K\setminus\left\{0\right\} $, thus we have
	\[ \lambda(f)=\frac{c_{f_{1,m},f_{1,m}}(1_{\Sp_{2n}(\R)})}{\check F_{f_{1,m}}(1_{\Sp_{2n}(\R)})}=\frac{\norm{f_{1,m}}^2_{H_m}}{1}\overset{\eqref{eq:085}}=C_{m,n},\qquad f\in(H_m)_K\setminus\left\{0\right\}. \]
	This proves \eqref{eq:048}, from which \eqref{eq:032} follows immediately.
\end{proof}

The proof of the following lemma is straightforward and left to the reader.

\begin{lemma}\label{lem:035}
	Let $ f:\calH_n\to\C $. Then, the Poincar\'e series
	\[ P_{\Gamma}f=\sum_{\gamma\in\Gamma}f\big|_m\gamma \]
	converges absolutely and uniformly on compact subsets of $ \calH_n $ (resp., almost everywhere on $ \calH_n $) if and only if the Poincar\'e series $ P_\Gamma F_f=\sum_{\gamma\in\Gamma}F_f(\gamma\spacedcdot) $ converges in the same way on $ \Sp_{2n}(\R) $. Moreover, assuming convergence, the following holds:
	\begin{enumerate}[label=\textup{(\roman*)}]
		\item\label{lem:035:1} $ F_{P_\Gamma f}=P_\Gamma F_f $.
		\item\label{lem:035:2} $ P_\Gamma f\equiv0 $ if and only if $ P_\Gamma F_f\equiv0 $.
	\end{enumerate}
\end{lemma}

The above results enable us to prove

\begin{theorem}\label{thm:069}
	Suppose that $ m\in\Z_{>2n} $. Let $ \Gamma $ be a congruence subgroup of $ \Sp_{2n}(\Z) $. For every $ \mu\in\C[X_{r,s}:1\leq r,s\leq n] $, the Poincar\'e series $ P_{\Gamma}f_{\mu,m} $ converges absolutely and uniformly on compact subsets of $ \calH_n $, and we have
	\begin{equation}\label{eq:036}
	S_m(\Gamma)=\left\{P_\Gamma f_{\mu,m}:\mu\in\C[X_{r,s}:1\leq r,s\leq n]\right\}. 
	\end{equation}
\end{theorem}

\begin{proof}
	For every $ \mu $, the function $ F_{\mu,m}:=F_{f_{\mu,m}} $ is a $ K $-finite matrix coefficient of the integrable representation $ \pi_m^* $  by \eqref{eq:032}, so the series $ P_\Gamma F_{\mu,m} $ converges absolutely and uniformly on compact sets (see the beginning of Sect.\ \ref{sec:039}), and hence so does the series $ P_\Gamma f_{\mu,m} $  by Lem.\ \ref{lem:035}. Next, we have 
	\[ \begin{aligned}
	\mathcal A_{cusp}(\Gamma\backslash\mathrm{Sp}_{2n}(\R))_{[\chi_{m}]}^{\mathfrak p_\C^-}&\overset{\eqref{eq:033}}=\left(L^2_{cusp}(\Gamma\backslash\Sp_{2n}(\R))_{[\pi_m^*]}\right)_{[\chi_m]}\\
	&\overset{\eqref{eq:024}}=\left\{P_\Gamma c_{h,h'}:h\in \left(H_m^*\right)_{[\chi_m]},\ h'\in \left(H_m^*\right)_K\right\}\\
	&\underset{\eqref{eq:038}}{\overset{\eqref{eq:037}}=}\left\{P_\Gamma c_{f_{1,m}^*,f_{\mu,m}^*}:\mu\in\C[X_{r,s}:1\leq r,s\leq n]\right\}\\
	&\overset{\eqref{eq:032}}=\left\{P_\Gamma F_{f_{\mu,m}}:\mu\in\C[X_{r,s}:1\leq r,s\leq n]\right\}.
	\end{aligned} \]
	By applying the inverse of the unitary isomorphism $ \Lambda_m $ of Prop.\ \ref{prop:034} to both sides of this equality and using Lem.\ \ref{lem:035}\ref{lem:035:1}, we obtain \eqref{eq:036}.
\end{proof}

\section{Non-vanishing of Poincar\'e series}\label{sec:081}

In this section, we study the non-vanishing of Poincar\'e series $ P_\Gamma f_{\mu,m} $ of Thm.\ \ref{thm:069} using the following result, which is a special case of the strengthened version \cite[Thm.\ 1]{zunar20} of Mui\'c's non-vanishing criterion \cite[Lem.\ 2.1]{muic11}.

\begin{lemma}\label{lem:042}
	Let $ \Gamma $ be a discrete subgroup of $ \Sp_{2n}(\R) $, and let $ \Lambda $ be a finite subgroup of $ \Gamma $.
	Let $ \varphi\in L^1(\Lambda\backslash\Sp_{2n}(\R)) $. Suppose that there exists a Borel measurable subset $ C\subseteq\Sp_{2n}(\R) $ with the following properties:
	\begin{enumerate}[label=\textup{(C\arabic*)}]
		\item\label{enum:042:0} $ \Lambda C=C $.
		\item\label{enum:042:1} $ CC^{-1}\cap\Gamma\subseteq\Lambda $.
		\item\label{enum:042:2} $ \int_C\abs{\varphi(g)}\,dg>\int_{\Sp_{2n}(\R)\setminus C}\abs{\varphi(g)}\,dg $.
	\end{enumerate}
	Then, the Poincar\'e series
	\[ P_{\Lambda\backslash\Gamma}\varphi=\sum_{\gamma\in\Lambda\backslash\Gamma}\varphi(\gamma\spacedcdot) \]
	converges absolutely almost everywhere on $ \Sp_{2n}(\R) $, and $ P_{\Lambda\backslash\Gamma}\varphi\in L^1(\Gamma\backslash \Sp_{2n}(\R))\setminus\left\{0\right\} $.
\end{lemma}

By the $ KAK $ decomposition \cite[Thm.\ 5.20]{knapp1986}, every $ g\in \Sp_{2n}(\R) $ has a factorization
\[ g=kh_tk' \]
with $ k,k'\in K $ and $ h_t=\mathrm{diag}(e^{t_1},\ldots,e^{t_n},e^{-t_1},\ldots,e^{-t_n}) $ for some $ t=(t_1,\ldots,t_n)\in\mathbb R^n $. This factorization is not unique: $ Kh_tK=Kh_{t'}K $ if and only if $ t $ and $ t' $ coincide up to permutations and sign changes of coordinates. For every $ t\in\R^n $, let $ d_t=\mathrm{diag}(t)\in M_n(\R) $.
In terms of the $ KAK $ decomposition, the matrix coefficients $ F_{\mu,m} $ are given by a simple formula:

\begin{lemma}
	Let $ m\in\Z_{>n} $ and $ \mu\in\C[X_{r,s}:1\leq r,s\leq n] $. For all $ u,u'\in\mathrm U(n) $ and $ t\in\R^n $, we have
	\begin{equation}\label{eq:045}
	F_{\mu,m}(k_uh_tk_{u'})=\det u^m\,\frac{\mu\left(u\,\tanh(d_t)\,u^\top\right)}{\det\cosh(d_t)^m}\,\det(u')^m.
	\end{equation}
\end{lemma}

\begin{proof}
	The proof is elementary: since $ F_{\mu,m}=F_{f_{\mu,m}} $, by \eqref{eq:013} and \eqref{eq:041} we have
	\[ F_{\mu,m}(k_uh_tk_{u'})=\left(p_\mu\big|_m\ell k_uh_tk_{u'}\right)(iI_n)\underset{\eqref{eq:086}}{\overset{\eqref{eq:074}}=}\frac{\mu(\ell k_uh_tk_{u'}.(iI_n))}{j(\ell k_uh_tk_{u'},iI_n)^m}. \]
	Writing $ u=A+iB $ and $ u'=A'+iB' $ with $ A,B,A',B'\in M_n(\R) $, $ \ell k_uh_tk_{u'}.(iI_n) $ equals
	\[ \begin{aligned}
		&\big(iA\exp(2d_t)+B-i\left(-iB\exp(2d_t)+A\right)\big)\big(iA\exp(2d_t)+B+i\left(-iB\exp(2d_t)+A\right)\big)^{-1}\\
		&=\big((A+iB)\sinh(d_t)\cdot 2i\exp(d_t)\big)\big((A-iB)\cosh(d_t)\cdot 2i\exp(d_t)\big)^{-1}\\
		&=u\,\tanh(d_t)\,u^\top,
	\end{aligned} \]
	and one shows similarly, using \eqref{eq:040}, that $ j(\ell k_uh_tk_{u'},iI_n)=\frac{\det\cosh(d_t)}{\det(uu')} $. The claim follows.
\end{proof}

We will apply Lem.\ \ref{lem:042} to the series $ P_\Gamma F_{\mu,m} $ using the set $ C $ of the form
\begin{equation}\label{eq:087}
 C_R=K\left\{h_t:t\in\left[0,R\right[^n\right\}K 
\end{equation}
for a suitable $ R\in\R_{>0} $. For a complex square matrix $ g=(g_{r,s}) $, let $ \norm g=\sqrt{\sum_{r,s}\abs{g_{r,s}}^2} $. The following lemma is a generalization of \cite[Lem.\ 6-20]{muic10}.
  
\begin{lemma}\label{lem:044}
	Let $ R\in\R_{>0} $ and $ g\in C_RC_R^{-1} $. Then,
	\[ \norm g<\sqrt{2n\cosh(4R)}. \]
\end{lemma}

\begin{proof}
	By \eqref{eq:087}, $ g=kh_tk_uh_{-t'}k' $ for some $ k,k'\in K $, $ t,t'\in[0,R[^n $, and $ u=A+iB\in\mathrm U(n) $. Since $ k,k'\in\mathrm O(2n) $, we have
	\begin{align*}
		\norm g^2&=\norm{h_tk_uh_{-t'}}=\norm{\begin{pmatrix}\exp(d_t)A\exp(-d_{t'})&\exp(d_t)B\exp(d_{t'})\\-\exp(-d_t)B\exp(-d_{t'})&\exp(-d_t)A\exp(d_{t'})\end{pmatrix}}^2\\
		&=\sum_{r,s}\left(\left(e^{2(t_r-t_s')}+e^{-2(t_r-t_s')}\right)a_{r,s}^2+\left(e^{2(t_r+t_s')}+e^{-2(t_r+t_s')}\right)b_{r,s}^2\right)\\
		&=2\sum_{r,s}\Big(a_{r,s}^2\cosh2(t_r-t_s')+b_{r,s}^2\cosh2(t_r+t_s')\Big)\\
		&<2\cosh(4R)\sum_{r,s}\left(a_{r,s}^2+b_{r,s}^2\right)=2\cosh(4R)\norm{u}^2=2n\cosh(4R).\qedhere
	\end{align*}
\end{proof}

On the other hand, we have the following

\begin{lemma}\label{lem:045}
	Let $ N\in\Z_{>0} $. For every $ \gamma\in\Gamma_n(N)\setminus K $,
	\[ \norm\gamma\geq\sqrt{N^2+2n}. \]
\end{lemma}

\begin{proof}
	The group $ \Gamma_n(N)\cap K=\Gamma_n(N)\cap\mathrm O(2n) $ obviously consists of those elements $ \gamma\in\Gamma_n(N) $ that have exactly one non-zero entry in each row (resp., column). Thus, if $ \gamma\in\Gamma_n(N)\setminus K $, then $ \gamma $ has at least $ 2n+1 $ non-zero (integer) entries. As at least one of those entries is off-diagonal and hence of absolute value at least $ N $, it follows that $ \norm\gamma\geq\sqrt{N^2+2n} $. 
\end{proof}

Let $ N\in\Z_{>0} $, and let $ \Gamma\subseteq\Gamma_n(N) $ be a congruence subgroup. By Lemmas \ref{lem:044} and \ref{lem:045}, the set $ C_R $ satisfies the condition \ref{enum:042:1} of Lem.\ \ref{lem:042} with $ \Lambda=\Gamma\cap K $ if the inequality
\[ \sqrt{N^2+2n}\geq\sqrt{2n\cosh(4R)} \]
or equivalently
\begin{equation}\label{eq:046}
\tanh^2R\leq\left(\sqrt{1+\frac{4n}{N^2}}+\sqrt{\frac{4n}{N^2}}\right)^{-2}=:M(N)
\end{equation}
holds.

Next, by \cite[Prop.\ 5.28]{knapp1986} there exists $ M_0\in\R_{>0} $ such that for every $ f\in C_c(\Sp_{2n}(\R)) $,
\[ \int_{\Sp_{2n}(\R)}f(g)\,dg=M_0\int_K\int_{A^+}\int_K f(kh_tk')\hspace{-1mm}\left(\prod_{1\leq r<s\leq n}\hspace{-3mm}\sinh(t_r-t_s)\right)\hspace{-1mm}\left(\prod_{1\leq r\leq s\leq n}\hspace{-3mm}\sinh(t_r+t_s)\right)dk\,dt\,dk', \]
where $ A^+=\left\{t=(t_1,\ldots,t_n)\in\R^n:t_1>\ldots>t_n>0\right\} $. From this formula and \eqref{eq:045}, after introducing the substitution $ x=\left(\tanh^2t_1,\ldots,\tanh^2t_n\right) $ it follows that for every $ R\in\R_{>0} $,
\begin{equation}\label{eq:047}
\int_{C_R}\abs{F_{\mu,m}(g)}\,dg=M_0\int_{A_{\tanh^2R}^+}\int_{\mathrm{U(n)}}\varphi_{\mu,m}(u,x)\,\mathrm du\,\mathrm dx, 
\end{equation}
where we use the notation $ A_t^+=\left\{x\in\R^n:t>x_1>\ldots>x_n>0\right\} $ for $ t\in\R_{>0} $, 
$ du $ is the unique Haar measure on $ \mathrm U(n) $ with respect to which $ \vol(\mathrm U(n))=1 $, and the function $ \varphi_{\mu,m}:\mathrm U(n)\times A_1^+\to\R_{\geq0} $ is given by
\[ \varphi_{\mu,m}(u,x)=\abs{\mu\left(u\,d_x^{\frac12}\,u^\top\right)}\,\left(\prod_{1\leq r<s\leq n}(x_r-x_s)\right)\prod_{r=1}^n(1-x_r)^{\frac m2-n-1}. \]

The function $ M:\Z_{>0}\to[0,1] $ defined by \eqref{eq:046} is non-decreasing, and $ \lim_{N\to\infty}M(N)=1 $.
Thus, by applying Lem.\ \ref{lem:042} to the Poincar\'e series $ P_\Gamma F_{\mu,m} $ with $ C=C_{\artanh\sqrt{M(N)}} $, by \eqref{eq:046}, \eqref{eq:047} and Lem.\ \ref{lem:035}\ref{lem:035:2} we obtain the following

\begin{proposition}\label{prop:043}
	Let $ m\in\Z_{>2n} $ and $ \mu\in\C[X_{r,s}:1\leq r,s\leq n]\setminus\left\{0\right\} $.
	Let $ N_0=N_0(\mu,m) $ denote the smallest integer $ N\in\Z_{>0} $ such that
	\begin{equation}\label{eq:049}
	\int_{A_{M(N)}^+}\int_{\mathrm U(n)}\varphi_{\mu,m}(u,x)\,du\,dx>\frac12\int_{A_1^+}\int_{\mathrm U(n)}\varphi_{\mu,m}(u,x)\,du\,dx.
	\end{equation}
	Then, for every  $ N\in\Z_{\geq N_0} $ and for every congruence subgroup $ \Gamma\subseteq\Gamma_n(N) $ such that $ \Gamma\cap K=\left\{1\right\} $, we have that $ P_\Gamma F_{\mu,m}\not\equiv0 $ and $ P_\Gamma f_{\mu,m}\not\equiv0 $.
\end{proposition}

We note that since $ \Gamma_n(N)\cap K=\left\{1\right\} $ for all $ N\in\Z_{\geq3} $, the condition $ \Gamma\cap K=\left\{1\right\} $ in Prop.\ \ref{prop:043} is relevant only if $ N\in\left\{1,2\right\} $. In particular:

\begin{corollary}\label{cor:081}
	Let $ m\in\Z_{>2n} $ and $ \mu\in\C[X_{r,s}:1\leq r,s\leq n]\setminus\left\{0\right\} $. Then,
	\[ P_{\Gamma_n(N)}F_{\mu,m}\not\equiv0\quad\text{and}\quad P_{\Gamma_n(N)}f_{\mu,m}\not\equiv0\quad\text{for}\quad N\gg0. \] 
\end{corollary}

In the case when $ \mu=\det^l $ for some $ l\in\Z_{\geq0} $, the function $ \varphi_{\mu,m} $ is given by
\begin{equation}\label{eq:053}
\varphi_{\det^l,m}(u,x)=\left(\prod_{r=1}^nx_r^{\frac l2}(1-x_r)^{\frac m2-n-1}\right)\prod_{1\leq r<s\leq n}(x_r-x_s)=:\phi_{l,m}(x),
\end{equation}
and
$ N_0(\det^l,m) $ of Prop.\ \ref{prop:043} can be defined as the smallest integer $ N\in\Z_{>0} $ such that
\begin{equation}\label{eq:054}
\int_{A_{M(N)}^+}\phi_{l,m}(x)\,dx>\frac12\int_{A_1^+}\phi_{l,m}(x)\,dx.
\end{equation}
Moreover, it follows from \eqref{eq:045} that
\begin{equation}\label{eq:051}
F_{\det^l,m}(kg)=\chi_{m+2l}(k)F_{\det^l,m}(g),\qquad k\in K,\ g\in\Sp_{2n}(\R). 
\end{equation}
Armed with this, we prove the following

\begin{proposition}\label{prop:082}
	Let $ m\in\Z_{>2n} $, $ l\in\Z_{\geq0} $, and $ N\in\Z_{>0} $. Then:
	\begin{enumerate}[label=\textup{(\roman*)}]
		\item\label{prop:054:1} $ P_{\Gamma_n(N)}F_{\det^l,m}\equiv0 $ and $ P_{\Gamma_n(N)}f_{\det^l,m}\equiv0 $ if one of the following holds:
		\begin{enumerate}[label=\textup{(V\arabic*)}]
			\item\label{enum:050:1} $ N=1 $ and $ 4\nmid m+2l $
			\item\label{enum:050:2} $ N=2 $ and $ 2\nmid m $.
		\end{enumerate}
		\item\label{prop:054:2} Suppose that neither \ref{enum:050:1} nor \ref{enum:050:2} holds, and define $ N_0(\det^l,m) $ as above. If $ N\geq N_0(\det^l,m) $, then $ P_{\Gamma_n(N)}F_{\det^l,m}\not\equiv0 $ and $ P_{\Gamma_n(N)}f_{\det^l,m}\not\equiv0 $.
	\end{enumerate}
\end{proposition}

\begin{proof}
	Let $ \mathcal S_n $ denote the group of permutation matrices in $ M_n(\R) $. We note that 
	\[ \Gamma_n(N)\cap K=\begin{cases}
	\big\{k_u:u\in\mathrm{Diag}_n\left(\left\{\pm1,\pm i\right\}\right)\mathcal S_n\big\},&\text{if }N=1\\
	\big\{k_u:u\in\mathrm{Diag}_n\left(\left\{\pm1\right\}\right)\big\},&\text{if }N=2\\
	\left\{1\right\},&\text{if }N\geq3.
	\end{cases} \]
	Thus, if \ref{enum:050:1} or \ref{enum:050:2} holds, then $ \chi_{m+2l}\big|_{\Gamma_n(N)\cap K} $ is a non-trivial character of $ \Gamma_n(N)\cap K $, hence
	\[ P_{\Gamma_n(N)\cap K}F_{\det^l,m}\overset{\eqref{eq:051}}=\left(\sum_{\gamma\in\Gamma_n(N)\cap K}\chi_{m+2l}(\gamma)\right)F_{\det^l,m}=0 \]
	and consequently $ P_{\Gamma_n(N)}F_{\det^l,m}=P_{(\Gamma_n(N)\cap K)\backslash\Gamma_n(N)}P_{\Gamma_n(N)\cap K}F_{\det^l,m}=0 $, which proves \ref{prop:054:1}.
	On the other hand, if neither \ref{enum:050:1} nor \ref{enum:050:2} holds, then $ \chi_{m+2l}\big|_{\Gamma_n(N)\cap K}=1 $, hence by \eqref{eq:051} the function $ F_{\det^l,m} $ is $ (\Gamma_n(N)\cap K) $-invariant on the left, so
	\[ P_{\Gamma_n(N)}F_{\det^l,m}=\abs{\Gamma_n(N)\cap K}\ P_{(\Gamma_n(N)\cap K)\backslash\Gamma_n(N)}F_{\det^l,m}, \]
	and by applying Lem.\ \ref{lem:042} to the Poincar\'e series $ P_{(\Gamma_n(N)\cap K)\backslash\Gamma_n(N)}F_{\det^l,m} $, with the set $ C $ exactly as in the proof of Prop.\ \ref{prop:043}, we obtain \ref{prop:054:2}.
\end{proof}

In the cases when $ n\in\left\{1,2\right\} $, the values of $ N_0(\det^l,m) $ for some small choices of $ l $ and $ m $ are given in Table \ref{table:101} below. These values were computed using Wolfram Mathematica 12.1 \cite{wolfram}; the code can be found in \cite{zunar22}.

\begin{table}[ht]
	\caption{Some values of $ N_0(\det^l,m) $ in the cases when $ n\in\{1,2\} $.}\label{table:101}
	\centering\footnotesize
	\begin{tabular}{|c|rrrrrrrr|}
		\hline
		\multicolumn{9}{|c|}{\vphantom{$ \bigg| $}\normalsize $ n=1 $}\\ [-0.5ex]
		\hline
		\backslashbox{$l$}{$m$}& 3 & 4 & 5 & 6 & 7 & 8 & 9 & 10  \\
		\hline
		0 & 14 & 6 & 4 & 4 & 3 & 3 & 3 & 2  \\
		1 & 23 & 9 & 6 & 5 & 4 & 4 & 3 & 3  \\
		2 & 32 & 12 & 8 & 6 & 5 & 5 & 4 & 4   \\
		3 & 40 & 15 & 10 & 7 & 6 & 5 & 5 & 4  \\
		4 & 49 & 18 & 11 & 9 & 7 & 6 & 5 & 5  \\
		5 & 58 & 21 & 13 & 10 & 8 & 7 & 6 & 6 \\
		6 & 67 & 24 & 15 & 11 & 9 & 8 & 7 & 6  \\
		7 & 75 & 26 & 16 & 12 & 10 & 8 & 7 & 7  \\
		8 & 84 & 29 & 18 & 13 & 11 & 9 & 8 & 7 \\
		9 & 93 & 32 & 20 & 15 & 12 & 10 & 9 & 8 \\
		10 & 102 & 35 & 22 & 16 & 13 & 11 & 9 & 8  \\
		11 & 111 & 38 & 23 & 17 & 14 & 12 & 10 & 9  \\
		12 & 119 & 41 & 25 & 18 & 15 & 12 & 11 & 10  \\
		\hline
	\end{tabular}\qquad
	\begin{tabular}{|c|rrrrrrrr|}
		\hline
		\multicolumn{9}{|c|}{\vphantom{$ \bigg| $} \normalsize$ n=2 $}\\ [-0.5ex]
		\hline
		\backslashbox{$l$}{m}& 5 & 6 & 7 & 8 & 9 & 10 & 11 & 12  \\ 
		\hline
		0 & 77 & 25 & 15 & 11 & 9 & 8 & 7 & 6 \\
		1 & 107 & 33 & 20 & 14 & 11 & 10 & 8 & 8\\
		2 & 137 & 41 & 24 & 17 & 14 & 11 & 10 & 9  \\
		3 & 167 & 49 & 28 & 20 & 16 & 13 & 11 & 10  \\
		4 & 197 & 58 & 33 & 23 & 18 & 15 & 13 & 11  \\
		5 & 227 & 66 & 37 & 26 & 20 & 17 & 14 & 12  \\
		6 & 257 & 74 & 41 & 29 & 22 & 18 & 16 & 14  \\
		7 & 287 & 82 & 46 & 32 & 24 & 20 & 17 & 15  \\
		8 & 317 & 90 & 50 & 34 & 26 & 22 & 18 & 16  \\
		9 & 347 & 98 & 54 & 37 & 29 & 23 & 20 & 17  \\
		10 & 377 & 107 & 59 & 40 & 31 & 25 & 21 & 18  \\
		11 & 407 & 115 & 63 & 43 & 33 & 27 & 22 & 19  \\
		12 & 437 & 123 & 67 & 46 & 35 & 28 & 24 & 21  \\
		\hline
	\end{tabular}
\end{table}

\section{Petersson inner products of certain automorphic forms}\label{sec:082}

In the case when $ n=1 $, the following proposition follows from \cite[Thm.\ 2-11]{muic12}.

\begin{proposition}\label{prop:065}
	Let $ \Gamma $ be a congruence subgroup of $ \Sp_{2n}(\Z) $, and let $ m\in\Z_{>2n} $.  Let $ \varphi\in\calA_{cusp}(\Gamma\backslash\Sp_{2n}(\R))_{[\chi_m]}^{\mathfrak p_\C^-} $. Then,
	\begin{equation}\label{eq:064}
		\scal\varphi{P_\Gamma F_{1,m}}_{L^2(\Gamma\backslash\Sp_{2n}(\R))}=C_{m,n}\,\varphi(1_{\Sp_{2n}(\R)}). 
	\end{equation}
\end{proposition}

\begin{proof}
	We adapt the argument of \cite[proof of Thm.\ 5.4]{zunar18}: Without loss of generality, suppose that $ \varphi\not\equiv0 $.
	For any $ F\in L^1(\Sp_{2n}(\R)) $ and $ \psi\in L^2(\Gamma\backslash\Sp_{2n}(\R)) $, $ R(F)\psi\in L^2(\Gamma\backslash\Sp_{2n}(\R)) $ is standardly defined by the condition
	\[ \scal{R(F)\psi}\phi_{L^2(\Gamma\backslash\Sp_{2n}(\R))}=\int_{\Sp_{2n}(\R)}F(h)\scal{R(h)\psi}\phi_{L^2(\Gamma\backslash\Sp_{2n}(\R))}\,dh \]
	for all $ \phi\in L^2(\Gamma\backslash\Sp_{2n}(\R)) $.
	It is well-known that
	\begin{equation}\label{eq:056}
		(R(F)\psi)(g)=\int_{\Sp_{2n}(\R)}F(h)\psi(gh)\,dh\qquad\text{for a.a.\ }g\in\Sp_{2n}(\R). 
	\end{equation}
	If $ \psi=\varphi $ and $ F=\overline{F_{1,m}} $, then the right-hand side of \eqref{eq:056} defines a continuous complex function of $ g\in\Sp_{2n}(\R) $ by the dominated convergence theorem  since $ \varphi $ is continuous and bounded, and $ F_{1,m} $ is in $ L^1(\Sp_{2n}(\R)) $ as a $ K $-finite matrix coefficient of the integrable representation $ \pi_m^* $ (see Prop.\ \ref{prop:057}\ref{prop:057:2} and Lem.\ \ref{lem:058}). In particular,
	\begin{equation}\label{eq:060}
		\begin{aligned}
		\left(R\left(\overline{F_{1,m}}\right)\varphi\right)(1_{\Sp_{2n}(\R)})&\overset{\phantom{\eqref{eq:088}}}=\int_{\Sp_{2n}(\R)}\overline{F_{1,m}(h)}\,\varphi(h)\,dh\\
		&\overset{\eqref{eq:088}}=\int_{\Gamma\backslash\Sp_{2n}(\R)}\sum_{\gamma\in\Gamma}\overline{F_{1,m}(\gamma h)}\,\varphi(h)\,dh\\
		&\overset{\phantom{\eqref{eq:088}}}=\scal\varphi{P_\Gamma F_{1,m}}_{L^2(\Gamma\backslash\Sp_{2n}(\R))}.
	\end{aligned}
	\end{equation}
	
	On the other hand, by Lem.\ \ref{lem:017}, $ \varphi $ spans the $ \chi_m $-isotypic component of an irreducible closed $ \Sp_{2n}(\R) $-invariant subspace $ H_\varphi $ of $ L^2(\Gamma\backslash\Sp_{2n}(\R)) $ that is unitarily equivalent to $ H_m^* $. Obviously, $ R\left(\overline{F_{1,m}}\right)\varphi\in H_\varphi $. More precisely, $ R\left(\overline{F_{1,m}}\right)\varphi\in (H_\varphi)_{[\chi_m]} $ since
	\[ 	\begin{aligned}
		\left(R\left(\overline{F_{1,m}}\right)\varphi\right)(gk)&\overset{\phantom{\eqref{eq:051}}}=\int_{\Sp_{2n}(\R)}\overline{F_{1,m}(h)}\,\varphi(gkh)\,dh\\
		&\overset{\phantom{\eqref{eq:051}}}=\int_{\Sp_{2n}(\R)}\overline{F_{1,m}(k^{-1}h)}\,\varphi(gh)\,dh\\
		&\overset{\eqref{eq:051}}=\overline{\chi_{m}(k^{-1})}\,\int_{\Sp_{2n}(\R)}\overline{F_{1,m}(h)}\,\varphi(gh)\,dh\\
		&\overset{\phantom{\eqref{eq:051}}}=\chi_{m}(k)\,\left(R\left(\overline{F_{1,m}}\right)\varphi\right)(g)
	\end{aligned}  \]
	for all $ g\in\Sp_{2n}(\R) $ and $ k\in K $. It follows that 
	\begin{equation}\label{eq:058}
		 R\left(\overline{F_{1,m}}\right)\varphi=\lambda\varphi 
	\end{equation}
	for some $ \lambda\in\C $. 
	
	By \eqref{eq:060} and \eqref{eq:058},
	\begin{equation}\label{eq:062}
		\scal\varphi{P_\Gamma F_{1,m}}_{L^2(\Gamma\backslash\Sp_{2n}(\R))}=\lambda\varphi(1_{\Sp_{2n}(\R)}).
	\end{equation}
	In order to compute $ \lambda $, we recall that by \cite[proof of Prop.\ 1.3.3]{wallachI}, the closed $ \Sp_{2n}(\R) $-invariant subspace 
	$ E_m=\left\{c_{h^*,f_{1,m}^*}:h\in H_m\right\} $
	 of $ (R,L^2(\Sp_{2n}(\R))) $ is unitarily equivalent to $ H_m^* $ via the assignment $ c_{h^*,f_{1,m}^*}\mapsto h^* $. Let us fix a unitary equivalence $ \Phi:H_\varphi\to E_m $. Since $ (H_\varphi)_{[\chi_m]}=\C\varphi $, and $ (E_m)_{[\chi_m]}=\C c_{f_{1,m}^*,f_{1,m}^*} $ is spanned by $ F_{1,m} $ by Prop.\ \ref{prop:057}\ref{prop:057:2}, we have
	\[ \Phi\varphi=\nu F_{1,m} \] 
	for some $ \nu\in\C^\times $. By applying $ \Phi $ to \eqref{eq:058}, we obtain an equality of continuous functions
	\[ R\left(\overline{F_{1,m}}\right)F_{1,m}=\lambda F_{1,m}, \]
	from which it follows that 
	\begin{equation}\label{eq:061}
		\lambda=\frac{\left(R\left(\overline{F_{1,m}}\right)F_{1,m}\right)(1_{\Sp_{2n}(\R)})}{F_{1,m}(1_{\Sp_{2n}(\R)})}=\frac{\norm{F_{1,m}}^2_{L^2(\Sp_{2n}(\R))}}{1}\overset{\eqref{eq:085}}=C_{m,n},
	\end{equation}
	where in the second equality, the equality of numerators is proved exactly as the first equality in \eqref{eq:060}, and the equality of denominators holds by \eqref{eq:045}. 
\end{proof}

By Prop.\ \ref{prop:034}, Lem.\ \ref{lem:035}\ref{lem:035:1}, and \eqref{eq:013}, Prop.\ \ref{prop:065} implies the following

\begin{corollary}\label{cor:079}
	Let $ \Gamma $ be a congruence subgroup of $ \Sp_{2n}(\Z) $, and let $ m\in\Z_{>2n} $. Then,
	\begin{equation}\label{eq:070}
		\scal f{P_\Gamma f_{1,m}}_{\Gamma\backslash\calH_n}=C_{m,n}\,f(iI_n),\qquad f\in S_m(\Gamma).
	\end{equation}
\end{corollary}

\section{Connection to the reproducing kernel for $ S_m(\Gamma) $}\label{sec:083}

Let $ \Gamma $ be a congruence subgroup of $ \Sp_{2n}(\Z) $, and let $ m\in\Z_{>2n} $. As an application of Cor.\ \ref{cor:079}, in Thm.\ \ref{thm:080} below we recover, for every $ \xi\in\calH_n $, a well-known formula for the Siegel cusp form $ \Delta_{\Gamma,m,\xi}\in S_m(\Gamma)$ that represents, in the sense of the Riesz representation theorem, the evaluation at $ \xi $ of Siegel cusp forms in $ S_m(\Gamma) $ (cf.\ \cite[Ch.\ 3, \S6, Thm.\ 1]{klingen}). In a sense, the Siegel cusp forms $ \Delta_{\Gamma,m,\xi} $ comprise the reproducing kernel, known already to Godement \cite[p.\ 30]{godement10}, of the Hilbert space $ S_m(\Gamma) $. We start with the following

\begin{lemma}\label{lem:066}
	Let $ \Gamma $ be a congruence subgroup of $ \Sp_{2n}(\Z) $ and $ g\in\Sp_{2n}(\Q) $. Let $ \Gamma^g=g\Gamma g^{-1} $. Then,  
	\[ \Gamma^g_\Z:=\Gamma^g\cap\Sp_{2n}(\Z) \]
	is a congruence subgroup of $ \Sp_{2n}(\Z) $, and we have 
	\begin{equation}\label{eq:067}
		\left\{f\big|_mg^{-1}:f\in S_m(\Gamma)\right\}\subseteq S_m(\Gamma^g_\Z). 
	\end{equation}
\end{lemma}

\begin{proof}
	The first claim holds by \cite[Ch.\ 2, (3.25)]{andrianov_zhuravlev}, 
	 and \eqref{eq:067} is proved in an elementary way using \eqref{eq:065}.
\end{proof}

We will also need the following lemma, whose proof is elementary and left to the reader.

\begin{lemma}
	Let $ \Gamma\subseteq\Gamma' $ be discrete subgroups of $ \Sp_{2n}(\R) $. Let $ f_1,f_2:\calH_n\to\C $ be Borel measurable functions such that $ f_r\big|_m\gamma=f_r $ for all $ \gamma\in\Gamma $ and $ r=1,2 $, and suppose that $ \scal{\abs{f_1}}{\abs{f_2}}_{\Gamma\backslash\calH_n}<\infty $. Then:
	\begin{enumerate}[label=\textup{(\roman*)}]
		\item For every $ g\in\Sp_{2n}(\R) $, we have
		\begin{equation}\label{eq:071}
			\scal{f_1}{f_2}_{\Gamma\backslash\calH_n}=\scal{f_1\big|_m g}{f_2\big|_mg}_{g^{-1}\Gamma g\backslash\calH_n}.
		\end{equation}
		\item If $ f_1\big|_m\gamma'=f_1 $ for all $ \gamma'\in\Gamma' $, then
		\begin{equation}\label{eq:072}
			\scal{f_1}{f_2}_{\Gamma\backslash\calH_n}=\scal{f_1}{\sum_{\gamma'\in\Gamma\backslash\Gamma'}f_2\big|_m\gamma'}_{\Gamma'\backslash\calH_n}.
		\end{equation}
	\end{enumerate}
\end{lemma}

\begin{theorem}\label{thm:080}
	Let $ m\in\Z_{>2n} $ and $ \xi=x+iy\in\calH_n $. We define the function $ f_{1,m,\xi}:\calH_n\to\C $,
	\begin{equation}\label{eq:089}
		f_{1,m,\xi}(z)=C_{m,n}^{-1}\,\det\left(\frac1{2i}\left(z-\overline\xi\right)\right)^{-m}.
	\end{equation}
	Let $ \Gamma $ be a congruence subgroup of $ \Sp_{2n}(\Z) $. Then, the Poincar\'e series
	\[ \Delta_{\Gamma,m,\xi}=P_\Gamma f_{1,m,\xi} \]
	converges absolutely and uniformly on compact subsets of $ \calH_n $ and defines a Siegel cusp form $ \Delta_{\Gamma,m,\xi}\in S_m(\Gamma) $ that satisfies
	\begin{equation}\label{eq:076}
		\scal f{\Delta_{\Gamma,m,\xi}}_{\Gamma\backslash\calH_n}=f(\xi),\qquad f\in S_m(\Gamma). 
	\end{equation}
\end{theorem}

\begin{proof}
	Let us write $ g=(n_xa_y)^{-1} $. One checks easily that
	\begin{equation}\label{eq:073}
		f_{1,m,\xi}=C_{m,n}^{-1}\,\det y^{-\frac m2}f_{1,m}\big|_mg,
	\end{equation}
	from which it follows that
	\begin{equation}\label{eq:075}
		\left(P_{\Gamma}F_{f_{1,m,\xi}}\right)(h)=C_{m,n}^{-1}\,\det y^{-\frac m2}\left(P_{g\Gamma g^{-1}}F_{1,m}\right)(gh),\quad h\in\Sp_{2n}(\R).
	\end{equation}
	Since $ F_{1,m} $ is a $ K $-finite matrix coefficient of the integrable representation $ \pi_m^* $  by Prop.\ \ref{prop:057}\ref{prop:057:2} and \eqref{eq:090}, the series $ P_{g\Gamma g^{-1}}F_{1,m} $ converges absolutely and uniformly on compact subsets of $ \Sp_{2n}(\R) $ and defines a bounded function on $ \Sp_{2n}(\R) $ (see the beginning of Sect.\ \ref{sec:039}), hence by \eqref{eq:075} so does the series $ P_{\Gamma}F_{f_{1,m,\xi}} $. By Lem.\ \ref{lem:035} and \eqref{eq:013}, it follows that the series $ P_{\Gamma}f_{1,m,\xi} $ converges absolutely and uniformly on compact subsets of $ \calH_n $ and satisfies the condition \ref{enum:002:2}. As it obviously satisfies \ref{enum:002:1}, it defines an element of $ S_m(\Gamma) $. 
	
	To prove \eqref{eq:076}, we note that if $ \xi\in\calH_n\cap M_n(\Q) $, then for every $ f\in S_m(\Gamma) $ we have
	\[ \begin{aligned}
		f(\xi)&\overset{\eqref{eq:074}}=\det y^{-\frac m2}\left(f\big|_mg^{-1}\right)(iI_n)\\
		&\underset{\eqref{eq:067}}{\overset{\eqref{eq:070}}=}C_{m,n}^{-1}\,\det y^{-\frac m2}\scal{f\big|_mg^{-1}}{P_{\Gamma^{g}_\Z} f_{1,m}}_{\Gamma^g_\Z\backslash\calH_n}\\
		&\overset{\eqref{eq:071}}=C_{m,n}^{-1}\,\det y^{-\frac m2}\scal{f}{P_{\left(\Gamma^{g}_\Z\right)^{g^{-1}}} \left(f_{1,m}\big|_mg\right)}_{\left(\Gamma^g_\Z\right)^{g^{-1}}\backslash\calH_n}\\
		&\overset{\eqref{eq:072}}=C_{m,n}^{-1}\,\det y^{-\frac m2}\scal{f}{P_{\Gamma} \left(f_{1,m}\big|_mg\right)}_{\Gamma\backslash\calH_n}\\
		&\overset{\eqref{eq:073}}=\scal{f}{P_{\Gamma} f_{1,m,\xi}}_{\Gamma\backslash\calH_n}.
	\end{aligned} \]
	This proves \eqref{eq:076} for $ \xi\in\calH_n\cap M_n(\Q) $. Since $ \calH_n\cap M_n(\Q) $ is dense in $ \calH_n $, the equality \eqref{eq:076} for general $ \xi\in\calH_n $ follows immediately by the continuity of the functions $ I_f:\calH_n\to\C $,
	\[ \begin{aligned}
		I_f(\xi)&\overset{\phantom{\eqref{eq:072}}}=\scal f{\Delta_{\Gamma,m,\xi}}_{\Gamma\backslash\calH_n}=\scal f{P_\Gamma f_{1,m,\xi}}_{\Gamma\backslash\calH_n}\\
		&\overset{\eqref{eq:072}}=\int_{\calH_n}f(z)\,\overline{f_{1,m,\xi}(z)}\,\det\Im(z)^m\,d\mathsf v(z)\\
		&\overset{\eqref{eq:089}}=C_{m,n}^{-1}\,(-2i)^{mn}\int_{\calH_n}f(z)\,\overline{\det\left(z-\overline\xi\right)}^{-m}\,\det\Im(z)^m\,d\mathsf v(z),
	\end{aligned} \]
	where $ f\in S_m(\Gamma) $. 
	Said continuity holds by the dominated convergence theorem, which is applicable since by \cite[Ch.\ III, \S6, (10)]{klingen}, for any choice of the height $ d\in\R_{>0} $ of the vertical strip
	\[ V_n(d)=\left\{\tau\in\calH_n:\mathrm{tr}(\Re(\tau)^2)\leq\frac1d,\ \Im(\tau)\geq dI_n\right\}, \]
	there exists $ \alpha_{n,d}\in\R_{>0} $ such that the following estimate holds for all $ z\in\calH_n $ and $ \xi\in V_n(d) $:
	\[ \abs{\det\left(z-\overline\xi\right)}^{-m}\leq\alpha_{n,d}\abs{\det(z+iI_n)}^{-m}.\qedhere \] 
\end{proof}

\bibliographystyle{amsplain}

\begin{thebibliography}{999999}
	\bibitem[AndZhu95]{andrianov_zhuravlev} Andrianov, A. N., Zhuravlev, V. G.: \textit{Modular forms and Hecke operators}. Transl.\ Math.\ Monogr.\  145, American Mathematical Society, Providence, RI (1995)
	
	\bibitem[AsgSch01]{asgari} Asgari, M., Schmidt, R.:	\textit{Siegel modular forms and representations.} Manuscripta Math.\ 104(2), 173--200 (2001)
	
	\bibitem[BorJac79]{bj}  Borel, A., Jacquet, H.: \textit{Automorphic forms and automorphic representations.} In: \textit{Automorphic forms, representations and L-functions} (Proc.\ Sympos.\ Pure Math.\ XXXIII, Oregon State Univ., Corvallis, Ore., 1977), Part 1, pp.\ 189--207, Amer.\ Math.\ Soc., Providence, R.I.\ (1979)
		
	
	\bibitem[Gel73]{gelbart} Gelbart, S.:	\textit{Holomorphic discrete series for the real symplectic group}.	Invent.\ Math.\ 19, 49--58 (1973)
	
	
	\bibitem[God57-6]{godement6} Godement, R.: \textit{Fonctions holomorphes de carr\'e sommable dans le demi-plan de Siegel}. S\'eminaire Henri Cartan, Tome 10, No. 1, Exposé No. 6 (1957--1958)
	
	\bibitem[God57-10]{godement10} Godement, R.: \textit{S\'erie de Poincar\'e et Spitzenformen}. S\'eminaire Henri Cartan, Tome 10, No. 1, Expos\'e No. 10 (1957--1958)
	
	\bibitem[Hall15]{hall}  Hall, B.: \textit{Lie groups, Lie algebras, and representations. An elementary introduction.} Second edition. Graduate Texts in Mathematics 222. Springer, Cham (2015)
	
	
	\bibitem[Har66]{harish1966} Harish-Chandra: \textit{Discrete series for semisimple Lie groups. II. Explicit determination of the characters.} Acta Math. 116, 1–111 (1966)
	
	\bibitem[HecSch76]{hecht_schmid} Hecht, H., Schmid, W.:	\textit{On integrable representations of a semisimple Lie group}. Math.\ Ann.\ 220(2), 147--149 (1976)
	
	\bibitem[Hua63]{hua63}  Hua, L.\ K.: \textit{Harmonic analysis of functions of several complex variables in the classical domains}. Transl.\ Math.\ Monogr.\ 6, American Mathematical Society, Providence, RI (1963)
	
	\bibitem[Hum80]{humphreys}  Humphreys, J.\ E.: \textit{Introduction to Lie algebras and representation theory.} Third printing, revised. Graduate Texts in Mathematics 9. Springer-Verlag, New York-Heidelberg-Berlin (1980)
	
	\bibitem[Kli90]{klingen}  Klingen, H.: \textit{Introductory Lectures on Siegel Modular Forms}. Cambridge Studies in Advanced Mathematics 20. Cambridge University Press, Cambridge (1990)
	
	\bibitem[Kna86]{knapp1986} Knapp, A.\ W.: \textit{Representation theory of semisimple groups. An overview based on examples.} Princeton Mathematical Series, 36. Princeton University Press, Princeton, NJ (1986)
	
	
	\bibitem[Kna96]{knapp1996} Knapp, A.\ W.: \textit{Lie groups beyond an introduction}. Progress in Mathematics, 140. Birkhäuser Boston, Inc., Boston, MA (1996)
	
	\bibitem[KnaVog95]{knapp_vogan} Knapp, A.\ W., Vogan, D.\ A., Jr.: \textit{Cohomological induction and unitary representations}. Princeton Mathematical Series, 45. Princeton University Press, Princeton, NJ (1995)
	
	
	
	\bibitem[Mui09]{muic09} Mui\'c, G.: \textit{On a construction of certain classes of cuspidal automorphic forms via Poincar\'e series}. Math.\ Ann.\ 343(1), 207--227 (2009)
	
	\bibitem[Mui10]{muic10} Mui\'c, G.:	\textit{On the cuspidal modular forms for the Fuchsian groups of the first kind.} J.\ Number Theory 130(7), 1488-1511 (2010)
	
	\bibitem[Mui11]{muic11} Mui\'c, G.: \textit{On the non-vanishing of certain modular forms}. Int.\ J.\ Number Theory 7(2), 351--370 (2011)
	
	\bibitem[Mui12]{muic12} Mui\'c, G.:	\textit{On the inner product of certain automorphic forms and applications}. J.\ Lie Theory 22(4), 1091--1107 (2012)
	
	\bibitem[Mui19]{muic19} Mui\'c, G.:	\textit{Smooth cuspidal automorphic forms and integrable discrete series}. Math.\ Z.\ 292(3--4), 895--922 (2019)
	
	\bibitem[Wall88]{wallachI} Wallach, N.\ R.: \textit{Real reductive groups. I.} Pure and Applied Mathematics, 132. Academic Press, Inc., Boston, MA (1988)
	
	\bibitem[Wol20]{wolfram} Wolfram Research, Inc., Mathematica, Version 12.1, Champaign, IL (2020)
	
	\bibitem[{\v Z}un18]{zunar18} \v Zunar, S.: \textit{On Poincar\'e series of half-integral weight}. Glas.\ Mat.\ Ser.\ III 53(2) 239--264 (2018)
	
	
	\bibitem[{\v Zun}20]{zunar20} \v Zunar, S.: \textit{On the non-vanishing of $ L $-functions associated to cusp forms of half-integral weight}. Ramanujan J.\ 51(3), 455--477 (2020) 
	
	\bibitem[{\v Zun}22]{zunar22} \v Zunar, S.: \textit{Addendum to the paper ``On a family of Siegel Poincar\'e
	series'': Wolfram Mathematica code for computing the constants $ N_0(\det^l,m) $ of Prop. 6.7(ii)}. \url{http://www2.geof.unizg.hr/~szunar/papers/siegel/Addendum.pdf} (2022). Accessed 13 Oct 2022
\end{thebibliography}
\linespread{.96}

\end{document}